\documentclass[final]{siamltex}
\usepackage{graphicx}
\usepackage{filecontents}
\usepackage{color}
\usepackage{amsfonts} 
\usepackage{eufrak}      
\usepackage{dsfont} 
\usepackage{mathrsfs}  
\usepackage{mathabx}
\usepackage{cite}
\usepackage{enumerate}

\newcommand{\EQ}{\begin{eqnarray}}
\newcommand{\EN}{\end{eqnarray}}
\newcommand{\EQQ}{\begin{eqnarray*}}
\newcommand{\ENN}{\end{eqnarray*}}
\newcommand{\nnum}{\nonumber}

\newcommand{\eps}				{{\epsilon}}

\newcommand{\R}		{{\mathds{R}}}
\newcommand{\er}[1]	{{(\ref{#1})}}

\title{Optimization Methods on Riemannian Manifolds via Extremum Seeking Algorithms \thanks{This work was supported by the Australian Research Council Discovery Project DP120101144.}}

\author{Farzin Taringoo$^{\dagger}$,  Peter M. Dower$^{\dagger}$, Dragan Ne\v{s}i\'{c}$^{\dagger}$ and Ying Tan \thanks{Department of Electrical and Electronic Engineering, The University of Melbourne, \{ftaringoo, pdower, dnesic, yint\}@unimelb.edu.au.}}

\begin{document}

\maketitle

\begin{abstract}
This  paper formulates the problem of Extremum Seeking for optimization of cost functions defined on Riemannian manifolds. We extend the conventional extremum seeking algorithms  for optimization problems in Euclidean spaces to optimization of cost functions defined on smooth Riemannian manifolds. This problem falls within the category of online optimization methods. We introduce the notion of geodesic dithers which is a perturbation of the optimizing trajectory in the tangent bundle of the ambient state manifolds and obtain the extremum seeking closed loop as a perturbation of the averaged gradient system. The main results are obtained by applying closeness of solutions and averaging theory on Riemannian manifolds. The main results are further extended for optimization on Lie groups.  Numerical examples on Riemannian manifolds (Lie groups) $SO(3)$ and $SE(3)$ are also presented at the end of the paper.

\end{abstract}



\pagestyle{myheadings}
\thispagestyle{plain}
\markboth{Farzin Taringoo,  Peter M. Dower, Dragan Ne\v{s}i\'{c} and Ying Tan}{Optimization Methods on Riemannian Manifolds via Extremum Seeking Algorithms}

\section{Introduction}

Optimization on manifolds is an important research area in optimization theory, see \cite{abs,udr, smith}. In this class of problems, the underlying optimization space is a manifold and consequently the analysis  differs from the standard optimization algorithms in Euclidean spaces. Numerical techniques and methods for optimization on manifolds should guarantee that in each step an optimizer is an element of the search space which is a manifold. Hence, optimization methods on manifolds are closely related to geometry of manifolds, see \cite{abs, udr, gabay, Luenberger}. 

As known, smooth manifolds can be embedded in high dimensional Euclidean spaces (Whitney Theorem) \cite{wit}. This means that optimization on manifolds can be carried out as constrained optimization problems in Euclidean spaces with sufficiently large dimensions. However, the corresponding embeddings for each particular manifold may not be available and algorithms developed on manifolds may be more efficient in terms of convergence speed and calculation burden \cite{ring}. The algorithms investigated in this field range from simple line search methods to more sophisticated algorithms such as trust region methods, see for example \cite{ yang, Bak}.

Optimization on manifolds can arise in a wide range of applications where the search space is constrained. Its applications may appear in signal processing \cite{man}, robotics \cite{hel}, and statistics \cite{fre}. The main underlying assumption in most of the numerical algorithms presented for optimization on manifolds is that the cost function is available. This makes the implementation of numerical algorithms simple since various numerical methods can be applied to calculate the sensitivity of cost functions to obtain numerical optimization trajectories. However, in many problems, cost functions may not be given in a well defined closed form. This necessitates generating a class of numerical methods which do not explicitly depend on the closed form of cost functions derivatives \cite{rom, shub}. This is the main motivation of this paper.  

Extremum seeking is a class of  on-line or real-time optimization methods for optimization of the 
steady-state  behavior of dynamical systems \cite{krs1}. This method is applied for optimization of both static functions and dynamical systems where the optimizer does not have a model of the cost/utility function or dynamical models. In other words, either cost functions or dynamical systems are unknown for the optimization procedure and the optimization algorithm should be able to converge to a vicinity of a local optimizer, see \cite{Khon, Guay2, guay, krs1, tan, nes6,nes5,nes55,teel, hans}. In this paper we only consider the extremum seeking algorithms for optimization of static cost functions and consequently we assume that cost functions are not available for the optimization procedure.

Extremum seeking finds its applications in a vast area of dynamical systems including robotics and mechanical systems. As is known mechanical systems are mathematically modeled on manifolds which do not necessarily possess vector space properties, see \cite{mar1, bloch, Lewis}. 
  Traditionally, extremum seeking systems have been analyzed in the class of unconstrained optimization methods on $\mathds{R}^{n}$ where the vector space properties of Euclidean spaces simplify the analysis. 
	
	In a more general framework, the underlying Euclidean spaces can be replaced by Riemannian manifolds. That is to say, we change an unconstrained optimization problem to a constrained one where the constraints are imposed by the ambient manifold spaces.  This necessitates a generalization of the extremum seeking framework for optimization on manifolds. To this end, we define a class of online optimization methods where the optimization trajectories lie on manifolds. As an example, the standard gradient descent and Newton methods are extended to their geometric versions by employing the notion of geodesics on Riemannian manifolds, see \cite{smith, Luenberger}. 
In this paper this step is done for extremum seeking algorithms by introducing the so-called \textit{geodesic dithers} which are the geometric versions of dither signals in standard extremum seeking framework, see \cite{krs1,tan}. By employing the geodesic dithers, we guarantee that during the optimization phase optimizing trajectories always lie on state manifolds. To analyze the behavior of the closed loop system, we employ averaging techniques  developed for dynamical systems on Riemannian manifolds and apply results of closeness of solutions to obtain closeness of optimizing trajectories to local optimizers.
	
	A recent version of extremum seeking algorithms for optimization of cost functions on submanifolds of Euclidean spaces appeared in \cite{hans2}. In \cite{hans2}, the authors analyzed an extremum seeking algorithm based on the Lie bracket approach. The method presented in \cite{hans2} is based upon embeddings of manifolds in Euclidean spaces and the techniques are inherited from extremum seeking algorithms in Euclidean spaces. However, in general, such embeddings may not be always available and implementation of extremum seeking algorithms on general Riemannian manifolds requires  geometric extensions of methods developed in Euclidean spaces.

In terms of exposition,  Section \ref{s2}  presents some  mathematical preliminaries needed for the analysis of the paper. Section \ref{s3} presents the extremum seeking problems on Riemannian manifolds and in Section \ref{slie} we extend the extremum seeking algorithm for optimization on Lie groups. In Sections \ref{s4} and \ref{s5} we present simple optimization examples on  Lie groups $SO(3)$ and $SE(3)$ by applying the extremum seeking methods developed in Section \ref{s3}. 
   
   \section{Preliminaries}\label{s2}In this section we provide the differential geometric material which is necessary for the analysis presented in the rest of the paper.  We define some of the frequently used symbols of this paper in Table \ref{table1}.\\
	\begin{table}[ht]
\label{table1}
\caption{Symbols and Their Descriptions} 
\centering 
\begin{tabular}{c c} 
\hline\hline 
Symbol& Description  \\ [0.5ex] 
\hline 
$M$ & Riemannian manifold \\ 
$G$ & Lie Group\\
$\star$ & Lie group operation\\
$\mathfrak{X}(M)$ & space of smooth time invariant\\ &vector fields on $M$\\ 
$\mathfrak{X}(M\times \mathds{R})$ & space of smooth time varying\\ &vector fields on $M$\\ 
$\nabla$ & Levi-Civita connection\\
$T_{x}M$ & tangent space at $x\in M$  \\
$TM$ & tangent bundle of $M$\\
$T^{*}_{x}M$ &  cotangent space at $x\in M$\\
$T^{*}M$ & cotangent bundle of $M$\\
$\frac{\partial}{\partial x_{i}}$ & basis tangent vectors at $x\in M$\\
$dx_{i}$ & basis cotangent vectors at $x\in M$\\
$f(x,t)$ & time varying vector fields on $M$  \\
$||\cdot||_{g^{M}}$ & Riemannian norm\\
$g^{M}(\cdot,\cdot)$ & Riemannian metric on $M$ \\ 
$d(\cdot,\cdot)$ & Riemannian distance on $M$\\
$C^{\infty}(M)$ & Space of smooth functions on $M$\\
$\Phi_{f}$ & flow associated with $f$\\
$TF$ & pushforward of $F$\\
$T_{x}F$ & pushforward of $F$ at $x$\\
$\R_{>0}$& $(0,\infty)$\\
$\R_{\geq 0}$& $[0,\infty)$\\
\hline 
\end{tabular}
\label{table:nonlin} 
\end{table}
\subsection{Riemannian manifolds}
\begin{definition}[see \cite{Lee2}, Chapter 3]
\label{def1}
A Riemannian manifold $(M,g^{M})$ is a differentiable manifold $M$ together with a Riemannian metric $g^{M}$, where $g^{M}$ is defined for each $x\in M$ via an inner product $g^{M}_{x}:T_x M\times T_x M\rightarrow\R$ on the tangent space $T_x M$ (to $M$ at $x$) such that the function defined by $x\mapsto g^{M}_x(X(x),Y(x))$ is smooth for any vector fields $X,Y\in\mathfrak{X}(M)$. In addition,
\begin{enumerate}[(i)]
\item $(M,g^{M})$ is $n$ dimensional if $M$ is $n$ dimensional;
\item $(M,g^{M})$ is connected if for any $x,y\in M$, there exists a piecewise smooth curve that connects $x$ to $y$. 
\end{enumerate}

\end{definition}

Note that in the special case where $M\doteq\mathds{R}^{n}$, the Riemannian metric $g^{M}$ is defined everywhere by $g^{M}_{x}\doteq\sum^{n}_{i=1}g_{ij}(x)dx_{i}\otimes dx_{i}$, $x\in M$, where $\otimes $ is the tensor product on $T^{*}_{x}M\times T^{*}_{x}M $ and $g_{ij}$ is the $(i,j)$ entity of $g^{M}_{x}$, see \cite{Lee2}.

As formalized in Definition \ref{def1}, connected Riemannian manifolds possess the property that any pair of points $x,y\in M$ can be connected via a path $\gamma\in\mathscr{P}(x,y)$, where
\EQ
	\mathscr{P}(x,y)
	\doteq \left\{ \gamma:[a,b]\rightarrow M \, \left| \, \begin{array}{c} 
			\gamma \mbox{ piecewise smooth,}
			\\
			\gamma(a) = x\,, \ \gamma(b) = y
		\end{array} \right. \right\}.
	\label{eq:paths}
\EN

\begin{theorem}[ \hspace{-.1cm}\cite{Lee3}, P. 94]
\label{thm:path}
Suppose $(M,g^{M})$ is an $n$ dimensional connected Riemannian manifold. Then, for any $x,y\in M$, there exists a piecewise smooth path $\gamma\in\mathscr{P}(x,y)$ that connects $x$ to $y$.

\end{theorem}

The existence of connecting paths (via Theorem \ref{thm:path}) between pairs of elements of an $n$ dimensional connected Riemannian manifold $(M,g^{M})$ facilitates the definition of a corresponding Riemannian distance. In particular, the Riemannian distance $d:M\times M\rightarrow\R$ is defined by the infimal path length between any two elements of $M$, with
\EQ\label{dist}
	d(x,y)
	& \doteq \inf_{\gamma\in\mathscr{P}(x,y)} \int_a^b
						\sqrt{g^{M}_{\gamma(t)} (\dot\gamma(t),\, \dot\gamma(t))}\, dt\,.
	\label{eq:distance-via-paths}
\EN
Note that since $\mathscr{P}(x,y)$ contains piecewise smooth paths connecting $x$ and $y$ then $\dot{\gamma}$ corresponds to left and right derivatives at non-smooth points of $\gamma$.

Using the definition of Riemannian distance $d$ of \er{eq:distance-via-paths}, it may be shown that $(M,d)$ defines a metric space, see \cite{Lee3}.
%
Next, the crucial concept of \textit{pushforward} operators is introduced.
\begin{definition}
For a given smooth mapping $F:M\rightarrow N$ from manifold $M$ to manifold $N$ the pushforward $TF$ is defined as a generalization of the Jacobian of smooth maps between Euclidean spaces, with
\EQ TF:TM\rightarrow TN,\nnum\EN
where 
\EQ T_{x}F:T_{x}M\rightarrow T_{F(x)}N,\nnum\EN
and 
\EQ T_{x}F(X_{x})\circ h=X_{x}(h\circ F),\hspace{.2cm}x\in M, X_{x}\in T_{x}M, h\in C^{\infty}(N).\nnum\EN
\end{definition}
  \subsection{Geodesic Curves}
   Geodesics are defined  \cite{jost} as length minimizing curves on Riemannian manifolds which satisfy
   \EQ \nabla_{\dot{\gamma}(t)}\dot{\gamma}(t)=0,\nnum\EN
   where $\gamma(\cdot)$ is a geodesic curve on $(M,g^{M})$ and $\nabla$ is the \textit{Levi-Civita} connection on $M$, see \cite{Lee3}.
    The solution of the Euler-Lagrange variational problem associated with the length minimizing problem shows that  all the geodesics on an $n$ dimensional Riemannian manifold  $(M,g^{M})$  must satisfy the following system of ordinary differential equations:
\EQ \label{geo}\ddot{\gamma}_{i}(s)+\sum^{n}_{j,k=1}\Gamma^{i}_{j,k}\dot{\gamma}_{j}(s)\dot{\gamma}_{k}(s)=0,\quad i=1,...,n,\EN
where
\EQ\label{cris} \Gamma^{i}_{j,k}=\frac{1}{2}\sum^{n}_{l=1}g^{il}(g_{jl,k}+g_{kl,j}-g_{jk,l}),\quad g_{jl,k}=\frac{\partial g_{jl}}{\partial x_{k}},\EN
where all the indices  $i,j,k,l$ run from $1$ up to $n=dim(M)$ and $[g^{ij}]\doteq[g_{ij}]^{-1}$. Recall that $g_{ij}$ is the $(i,j)$ entity of the matrix $g^{M}_{x}$.
 \begin{definition}[ \hspace{-.1cm}\cite{Lee3}, P. 72]
\label{d1}
 The restricted exponential map is defined by 
 \EQ \exp_{x}:T_{x}M\rightarrow M,\hspace{.2cm}\exp_{x}(v)=\gamma_{v}(1), v\in T_{x}M,\nnum\EN
 where $\gamma_{v}(1)$ is the unique maximal geodesic \cite[P. 59]{Lee3} initiating from $x$ with the velocity $v$ up to one as a solution of (\ref{geo}).
 \end{definition}

 Throughout, \textit{restricted exponential maps} are referred to as \textit{exponential maps}.
An open ball of radius $\delta>0$ and centered at $0\in T_x M$ in the tangent space at $x$ is denoted 
by $B_{\delta}(0)\doteq\{v\in T_{x}M\hspace{.1cm}|\hspace{.2cm}||v||_{g^{M}}<\delta\}$. Similarly, the corresponding closed ball is denoted by $\overline{B}_\delta(0)$. Using the local  diffeomorphic property of exponential maps, the corresponding geodesic ball centered at $x$ is defined as follows. 
 
 \begin{definition}
 For a vector space $V$, a \textit{star-shaped neighborhood} of $0\in V$ is any open set $U$ such that if $u\in U$ then $\alpha u\in U, \alpha\in[0,1]$.
 \end{definition}
 \begin{definition}[ \hspace{-.1cm}\cite{Lee3}, p. 76]
 A normal neighborhood around $x\in M$ is any open neighborhood  of $x$ which is a diffeomorphic image of a star shaped neighborhood of $0\in T_{x}M$ under the exponential map $\exp_{x}$.
 \end{definition}
 \begin{lemma}[ \hspace{-.1cm}\cite{Lee3}, Lemma 5.10]
 \label{eun}
 For any $x\in M$, there exists a neighborhood $B_{\delta}(0)$ in $T_{x}M$ on which $\exp_{x}$ is a diffeomorphism onto $\exp_{x}(B_{\delta}(0))\subset M$. 
 \end{lemma}
 \begin{definition}[\hspace{-.01cm}\cite{Lee3}, p. 76]
 In a neighborhood of  $x\in M$, where $\exp_{x}$ is a local diffeomorphism (this neighborhood always exists by Lemma \ref{eun}), a geodesic ball of radius $\delta>0$ is denoted by $\exp_{x}(B_{\delta}(0))\subset M$. The corresponding closed geodesic ball is denoted by $\exp_{x}(\overline{B}_{\delta}(0))$. 
 \end{definition}
 

 \begin{definition}\label{inj}The injectivity radius of $M$ is  
 \EQ \iota(M)\doteq \inf_{x\in M}\iota(x),\nnum\EN
 where
\EQ&&  \iota(x) \doteq \sup\{ r\in\R_{\ge0}| \exp_x \mbox{is diffeomorphic onto} \exp_x (B_r(0))\}.\nnum\EN

\end{definition}
\begin{definition}
The metric ball with respect to $d$ on $(M, g^{M})$ is defined by
\EQ B(x,r)\doteq \{ y\in M\hspace{.1cm}|\hspace{.2cm}d(x,y)<r\}.\nnum\EN
\end{definition}

The following lemma reveals a relationship between normal neighborhoods and metric balls on $(M,g^{M})$.
\begin{lemma}[ \hspace{-.1cm}\cite{Pet}, p. 122]
\label{lpp}
Given any $\eps\in\R_{>0}$ and $x\in M$, suppose that $\exp_x$ is a diffeomorphism  on $B_{\epsilon}(0)\subset T_{x}M$. If $B(x,r)\subset \exp_{x}B_{\epsilon}(0)$ for some $r\in \mathds{R}_{>0}$, then 
\EQ \exp_{x}B_{r}(0)=B(x,r).\nnum\EN
\hspace*{7.5cm}\end{lemma}

We note that $B_{\epsilon}(0)$ is the metric ball of radius $\epsilon$ with respect to the Riemannian metric $g^{M}$ in $T_{x}M$.
This paper focuses on dynamical systems governed by differential equations on a connected $n$ dimensional Riemannian manifold $(M,g^{M})$. Locally these differential equations are defined by (see \cite{Lee2})
\EQ &&\hspace{-0cm}\dot{x}(t)=f(x(t),t),\hspace{.2cm} f\in \mathfrak{X}(M\times \mathds{R}),\nnum\\&&\hspace{-0cm}\hspace{.2cm} x(0)=x_{0}\in M, t\in[t_{0},t_{f}]\subset \mathds{R}.\nnum\EN
The time dependent flow associated with a differentiable time dependent vector field $f$ is a map $\Phi_{f}$ satisfying 
\EQ \label{flow} &&\Phi_{f}:[t_0, t_{f}]\times [t_{0}, t_{f}]\times M\rightarrow M, \nnum\\&& (s_{f},s_{0},x)\mapsto \Phi_{f}(s_{f},s_{0},x)\in M,\EN
and
\EQ \frac{d\Phi_{f}(s,s_{0},x)}{ds}|_{s=t}=f(\Phi_{f}(t,s_{0},x),t).\nnum\EN 
One may show, for a smooth vector field $f$, that the integral flow $\Phi_{f}(s,t_{0},\cdot):M\rightarrow M$ is a local diffeomorphism , see \cite{Lee2}.

\subsection{Lie groups}
As is well-known a Lie group $(G,\star)$ is a Riemannian manifold equipped with  operations $g_{1}\star g_{2}$ and $g^{-1}$ which are smooth in their topologies ($\star$ is the group operation of $G$), see \cite{Varad, knap}. We recall that the Lie algebra $\mathcal{L}$ of a Lie group $G$ (see \cite{Lewis},\cite{ Varad}) is the tangent space at the identity element $e$ with the associated Lie bracket defined on the tangent space $\mathcal{L}\doteq T_e G$ of $G$. 
A vector field $X$ on $G$ is called \textit{left invariant } if
\EQ \forall g_{1},g_{2}\in G,\quad X(g_{1}\star g_{2})=T_{g_{2}}g_{1}X(g_{2}),\nnum\EN
where $g\star:G\rightarrow G,\hspace{.2cm}g\star(h)=g\star h,\hspace{.2cm} T_{g_{2}}g:T_{g_{2}}G\rightarrow T_{g\star g_{2}}G$ which immediately implies $X(g\star e)=X(g)=T_{e}L_{g}X(e)$.
For a left invariant vector field $X$, we define the \textit{exponential map} on Lie groups as follows:
\EQ\label{ex} \exp:\mathcal{L}\rightarrow G,\quad \exp(tX(e))\doteq\Phi(t,X),t\in \mathds{R},\EN

where $\Phi(t,X)$ is the solution of $\dot{g}(t)=X(g(t))$ with the boundary condition $g(0)=e$. It may be shown that the solution of $\dot{g}(t)=X(g(t))$ with initial condition $g_{0}\in G$ is given by $g_{0}\star \exp(tX(e))$  where $\star$ is the group operation of $G$, see \cite{Lewis}.
A Riemannian metric $g^{G}$ on a Lie group $G$ is left invariant if 
\EQ g^{G}_{g_{2}}(X(g_{2}),Y(g_{2}))=g^{G}_{g_{1}\star g_{2}}(T_{g_{2}}g_{1}X(g_{2}),T_{g_{2}}g_{1}Y(g_{2})), \hspace{.2cm}X,Y\in \mathfrak{X}(G).\nnum\EN
Analogous to left invariant metrics,  right invariant Riemannian metrics on $G$ are defined. A Riemannian metric which is both left and right invariant is called \textit{bi-invariant}. The Levi-Civita connection corresponding to a left invariant Riemannian metric $g^{G}$ is denoted by $\nabla^{\mathds{G}}$. It may be shown that the Levi-Civita connection of a left invariant metric is left invariant i.e. (see \cite{Lewis}) 
\EQ Tg\left(\nabla^{\mathds{G}}_{X}Y\right)=\nabla^{\mathds{G}}_{TgX}TgY.\nnum\EN
Note that the pushforward $Tg$ is not evaluated at any point, hence, $TgX$ is a well defined vector field on $G$. 

 The following lemma gives a relationship between the exponential maps  (\ref{ex}) and geodesics on Lie groups.
\begin{lemma}[\hspace{-.01cm}\cite{pen}]
\label{lg}
Assume $G$ is a Lie group which admits a bi-invariant Riemannian metric. Then, for a left invariant vector field $X$ on $G$ we have
\EQ \exp(tX(e))=\exp_{e}(t X),\nnum\EN
where $\exp(tX(e))$ is the exponential map (\ref{ex}) and $\exp_{e}(t X)$ is the geodesic emanating from $e$ by velocity $X(e)$. 
\end{lemma}

Note that $\exp(tX(e))$  is the solution of the left invariant vector field $X$ on $G$, whereas $\exp_{e}(t X)$ is the solution of (\ref{geo}) with the initial conditions $\exp_{e}(0X)=e$, $\frac{d}{dt}\exp_{e}(t X)|_{t=0}=X(e)$. 

   \section{Optimization on Manifolds and Extremum Seeking}
\label{s3}
Let us consider the optimization of a smooth function $J:M\rightarrow \mathds{R}_{\geq 0}$, where $(M,g^{M})$ is an $n$ dimensional smooth Riemannian manifold. Extremum seeking algorithms are a class of online optimization methods developed for minimizing/maximizing  smooth functions defined on Euclidean spaces. These methods can be applied to both static and dynamic functions. In this paper we restrict our analysis to static  functions defined on Riemannian manifolds. An extremum seeking closed loop for $M=\mathds{R}$ is shown in Figure \ref{12}. This is the simplest form of the extremum seeking algorithm to minimize/maximize a scalar function $J:\mathds{R}\rightarrow \mathds{R}_{\geq 0}$, see \cite{tan}. 
\begin{figure}
\begin{center}
\vspace*{-.4cm}
\hspace*{-1cm}
\includegraphics[scale=.3]{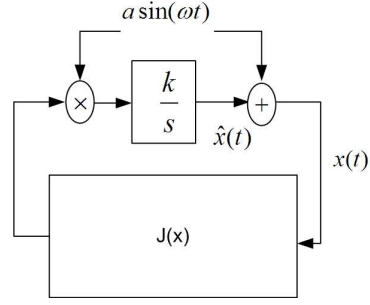}
 \caption{  Extremum seeking closed loop system for extremizing $J:M\rightarrow\R_{\ge 0}$, $M=\R$.}
      \label{12}
      \end{center}
   \end{figure} 
   The dither signal $a\sin(\omega t)$ provides a variation of the searching signal $\hat{x}(t)$ in the one dimensional space $\mathds{R}$. 
  The output $x(t)\in\R$ of the extremum seeking controller at time $t$ is $x(t) = \hat x(t) + a\, \sin(\omega
\, t)$, where $\hat x(t)\in\R$ is the corresponding output of the integrator shown. The closed loop dynamics in $\hat{x}$ coordinates are described by
  \EQ\label{kooni} \dot{\hat{x}}(t)&=&k a\sin(\omega t)J(\hat{x}(t)+a\sin(\omega t)).\EN
   The next lemma shows that, on average, the extremum seeking scheme of Figure \ref{12} is a perturbation of the gradient algorithm. 
  \begin{lemma}[\hspace{-.001cm}\cite{tan}]
  \label{l1}
  Consider the extremum seeking scheme in Figure \ref{12}. 
  Then, on average, the closed loop of the extremum seeking algorithm in Figure \ref{12} is a perturbation of the gradient algorithm in $\hat{x}$ coordinates.
  \end{lemma}
	
  The proof is based on the Taylor expansion of the cost function $J$ in $\hat{x}$ coordinates. By fixing $\hat{x}(t)$ to a dummy variable $z$, we have
  \EQ && J(z+a\sin(\omega t))= J(z)+\frac{\partial J}{\partial \hat{x}}|_{\hat{x}=z}a\sin(\omega t)+O(a^{2}).\nnum\EN
 Hence, the dynamical equations in $\hat{x}$ coordinates are given by
  \EQ\label{kir2}\dot{\hat{x}}(t)&=&ka\sin(\omega t)\Big(J(\hat{x})+\frac{\partial J}{\partial \hat{x}}|_{\hat{x}=\hat{x}(t)}a\sin(\omega t)+O(a^{2})\Big). \EN
  The dynamical equation (\ref{kir2}) is a periodic time-varying system where one may apply the averaging techniques, see \cite{Kha}. In particular, the averaged system is given by
  \EQ\label{kos} \dot{\hat{x}}(t)&=&\frac{1}{T}\int^{T}_{0}ka\sin(\omega t)\Big(J(\hat{x})+\frac{\partial J}{\partial \hat{x}}|_{\hat{x}=\hat{x}(t)}a\sin(\omega t)+O(a^{2})\Big)dt\nnum\\&=&\frac{ka^{2}}{2}\frac{\partial J}{\partial \hat{x}}|_{\hat{x}=\hat{x}(t)}+O(a^{4}),\EN
  where $T$ is the period of $\sin(\omega t)$. Obviously (\ref{kos}) is in a perturbation form of the gradient algorithm in $\mathds{R}$. 
  The results of Lemma \ref{l1} are of great importance since convergence of the averaged dynamical system in (\ref{kos}) to a neighborhood of an equilibrium (local minimum or maximum) is required in order to guarantee the closeness of solutions of the time varying system (\ref{kir2}) to a local optimizer, see \cite{krs1}. Here, we extend the framework above for optimization of cost functions defined on finite dimensional Riemannian manifolds. The main challenge is to introduce a class of dither signals which perturb the optimizer $\hat{x}$ without violating the restrictions imposed by the ambient manifolds. This is done by employing the so-called \textit{geodesic dithers} as follows.

	Consider an $n$ dimensional Riemannian manifold $(M,g^{M})$. For any $x\in M$, we consider the following local time-varying perturbation
  \EQ \label{gd}x_{p}(t)=\exp_{x}\left(\sum^{n}_{i=1}a_{i}\sin(\omega_{i}t)\frac{\partial}{\partial x_{i}}\right),\hspace{.2cm}0<a_{i},\EN 
  where $\{\frac{\partial}{\partial x_{i}}\},\hspace{.2cm}i=1,\cdots,n$, is the basis for the tangent space at $x$. As formalized in Definition \ref{d1}, $\exp_{x}v$ is a geodesic emanating from $x\in M$ with velocity $v\in T_{x}M$. In this case we perturb the $n$ different coordinates on $M$ with different frequencies $\omega_{i},\hspace{.2cm}i=1,\cdots,n$.
  As an example, for a one dimensional Riemannian manifold $S^{1}$, Figure \ref{13} shows the example of a geodesic dither $\exp_{x}a\sin(\omega t)\frac{\partial}{\partial \theta}$ at $x\in S^{1}$ where a local coordinate system corresponding to $S^{1}$ is given by (for the definition of coordinate systems see \cite{Lee4})
\EQ \psi:\theta\rightarrow (\sin(\theta),\cos(\theta))\in \mathds{R}^{2}, \hspace{.2cm}\theta\in (0,2\pi)\subset \mathds{R}.\nnum\EN 
\begin{figure}
  \vspace*{0cm}\begin{center}
\hspace*{-.25cm}\includegraphics[scale=.35]{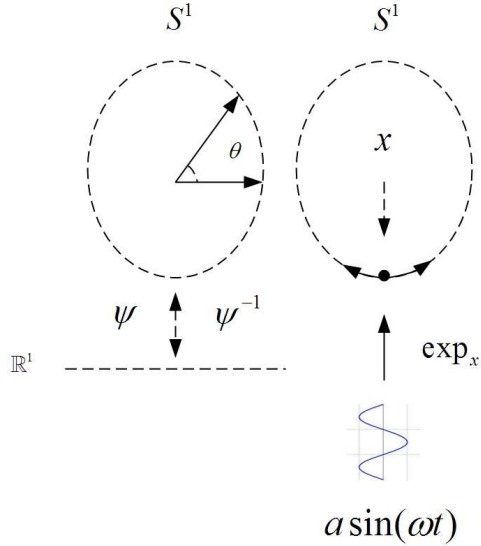}
 \caption{Geodesic dither at $x$ on $S^{1}$.}
     \label{13}
     \end{center}
  \end{figure}
Motivated by the classical extremum seeking closed loop of Figure \ref{12}, we present a  time-varying extremum seeking  vector field $f(\hat{x},t)\in T_{\hat{x}}M$ for optimization on $(M,g^{M})$ which is locally given by
  \EQ \label{kk}&&\hspace{-.5cm}f(\hat{x},t)\doteq k \sum^{n}_{i=1}a_{i}\sin(\omega_{i}t)J\left(\exp_{\hat{x}}\sum^{n}_{j=1}a_{j}\sin(\omega_{j}t)\frac{\partial}{\partial x_{j}}\right)\frac{\partial}{\partial x_{i}}.\EN
	In this paper we assume that the optimization problem is to minimize a cost function, hence, following (\ref{kos}), without loss of generality assume $k=-1$. 
  Finally, the optimizing trajectory $\hat{x}(\cdot)$ is a solution of the time dependent differential equation 
  \EQ \label{kkir}\dot{\hat{x}}(t)=f(\hat{x}(t),t)\in T_{\hat{x}(t)}M.\EN
	The closed loop system (\ref{kkir}) is called the \textit{extremum seeking system} on $(M,g^{M})$.
	Note that $t$ appears as a parameter in $\exp_{\hat{x}}\left(\sum^{n}_{i=1}a_{i}\sin(\omega_{i}t)\frac{\partial}{\partial x_{i}}\right)$. That is to say 
	\EQ \exp_{\hat{x}}\left(\sum^{n}_{i=1}a_{i}\sin(\omega_{i}t)\frac{\partial}{\partial x_{i}}\right)=\exp_{\hat{x}}\left(\eta\sum^{n}_{i=1}a_{i}\sin(\omega_{i}t)\frac{\partial}{\partial x_{i}}\right)\left|_{\eta=1},\right.\nnum\EN
	where $\eta$ and $t$ are independent. Also note that the optimization algorithm (\ref{kk}) does not require any information about the gradient of $J$. However, cost function $J$ should be measurable for the optimizing algorithm. 
	
  The next lemma proves that on compact Riemannian manifolds one may choose parameters $a_{i}>0$ sufficiently small such that for all $x\in M$ we have\\ $\exp_{x}\left(\sum^{n}_{i=1}a_{i}\sin(\omega_{i}t)\frac{\partial}{\partial x_{i}}\right)\in \iota(x)$. These results will be employed to obtain the Taylor expansion of cost functions on Riemannian manifolds along geodesics. 
  \begin{lemma}
  \label{ll1}
  Consider the geodesic dither introduced by (\ref{gd}) on a smooth $n$ dimensional compact Riemannian manifold $(M,g^{M})$. Then for all $x\in M$, we may select $a_{i}>0,\hspace{.2cm}i=1,\cdots,n$, such that for all $t\in \mathds{R}$, $\exp_{x}\left(\sum^{n}_{i=1}a_{i}\sin(\omega_{i}t)\frac{\partial}{\partial x_{i}}\right)\in \iota(x)$.
  \end{lemma}
  \begin{proof}
  Since $M$ is compact and smooth  then $\iota(M)$ is bounded from below, see \cite{Klin}. Hence, for all $x\in M$ there exists $\kappa\in \mathds{R}_{>0}$ such that $\kappa<\iota(x)$. The Riemannian norm of the dither signal is given by 
  \EQ &&\left\|\sum^{n}_{i=1}a_{i}\sin(\omega_{i}t)\frac{\partial}{\partial x_{i}}\right\|^{2}_{g^{M}}=
  \sum^{n}_{i,j=1}a_{i}a_{j}\sin(\omega_{i}t)\sin(\omega_{j}t)g^{M}\left(\frac{\partial}{\partial x_{i}},\frac{\partial}{\partial x_{j}}\right).\nnum\EN
  Since $M$ is compact, $\sum^{n}_{i,j=1}g^{M}(\frac{\partial}{\partial x_{i}},\frac{\partial}{\partial x_{j}})$ attains its maximum on $M$. Hence, $a_{i}>0,\hspace{.2cm}i=1,\cdots,n$, may be selected sufficiently small such that $||\sum^{n}_{i=1}a_{i}\sin(\omega_{i}t)\frac{\partial}{\partial x_{i}}||^{2}_{g^{M}}\leq \kappa^{2},\hspace{.2cm}\forall t\in \mathds{R},$ which completes the proof.
   \end{proof}
	
  We adopt the following assumption for the cost function $J$ on $(M,g^{M})$ and the dither frequencies $\omega_{i},\hspace{.2cm}i=1,\cdots,n$. This assumption is compatible with the main assumption on dither frequencies in \cite{fri} for multi-agent extremum seeking algorithms. 
  \newtheorem{assumption}{Assumption}
  \begin{assumption}
  \label{as}
  Cost function $J:M\rightarrow \mathds{R}_{\geq 0}$ is smooth and locally positive definite in a neighborhood of a unique local minimum $x^{*}\in M$, where $J(x^{*})=0$. The dither frequencies are $\omega_{i}=\omega\bar{\omega}_{i}$, where $\bar{\omega}_{i}$ is rational,  $\bar{\omega}_{i}\ne\bar{\omega}_{j}$, $2\bar{\omega}_{i}\ne \bar{\omega}_{j},j\ne i$ and $\bar{\omega}_{i}\ne \bar{\omega}_{j}+\bar{\omega}_{k}$ for distinct $i,j,k$, where $\omega_{i},\omega\in \mathds{R}_{>0}, i,j\in 1,\cdots,n$. 
  \end{assumption}
	
	Here we introduce the \textit{gradient and average systems} which correspond to the extremum seeking vector field (\ref{kk}) on $(M,g^{M})$. For the smooth function $J:M\rightarrow \mathds{R}_{\geq 0}$ the gradient system is defined by
	\EQ \label{ggs}\dot{x}(t)=-\sum^{n}_{i=1}\nabla_{\frac{\partial}{\partial x_{i}}}J(x(t))\frac{\partial}{\partial x_{i}},\EN
	where the set $\{\frac{\partial}{\partial x_{i}}, i=1,\cdots,n\}$ is a basis of $T_{x}M$.
	We note that the formal definition of the gradient $grad J$ of $J$ is given by \cite{jost} as 
	\EQ \label{kosak}dJ(X)=g^{M}(\mbox{grad}J,X),\hspace{.2cm}X\in \mathfrak{X}(M),\EN
	where $dJ$ is the one form differential of $J$  locally given by $dJ=\sum^{n}_{i=1}\frac{\partial J}{\partial x_{i}}dx_{i}\in T^{*}_{x}M$. 
	Note that the existence of $\mbox{grad}J$ in (\ref{kosak}) is implied by an application of Riesz representation theorem since $dJ$ belongs to the dual space $T^{*}_{x}M$ and $g^{M}(\cdot,\cdot)$ defines an inner product on $T_{x}M$, see \cite{rudin}. 
	In this case 
	\EQ \mbox{grad}J(x)=\sum^{n}_{i,j=1}g^{ij}(x)\nabla_{\frac{\partial}{\partial x_{i}}}J(x)\frac{\partial}{\partial x_{j}},\nnum\EN
	where $[g^{ij}]=[g_{ij}]^{-1}$. Hence, the formal gradient system corresponding to $J$ is $\label{gs}\dot{x}(t)=-\sum^{n}_{i=1}g^{ij}(x(t))\nabla_{\frac{\partial}{\partial x_{i}}}J(x(t))\frac{\partial}{\partial x_{i}}$. However, in this paper, we adopt the terminology that the gradient system of $J$ refers to (\ref{ggs}). The \textit{scaled} version of the gradient system (\ref{ggs}) is given as 
	\EQ\label{sc} \dot{x}(t)= \acute{f}(x(t))\doteq-\sum^{n}_{i=1}\frac{a^{2}_{i}}{2}\nabla_{\frac{\partial}{\partial x_{j}}}J(x(t))\frac{\partial}{\partial x_{i}}, a_{i}>0.\EN
	With no further confusion we refer the scaled gradient system as gradient system. 
	For the periodic time varying vector field $f(x,t)$ in (\ref{kk}), the averaged dynamical system is defined as follows
	\EQ \label{av}\dot{x}(t)= \hat{f}(x(t))\doteq\frac{1}{T}\int^{T}_{0}f(x(t),\tau)d\tau,\EN
	where $T$ is the period of $f$, i.e. $f(x,t)=f(x,t+T)$.
	
	The following lemma proves that, on average, the closed loop of the extremum seeking system  (\ref{kkir}) is a perturbation of  the gradient system of the cost function $J$.
	\begin{lemma}
	\label{kir}
	Consider the extremum seeking system in (\ref{kkir}) on a compact Riemannian manifold $(M,g^{M})$ where the optimizing trajectory is perturbed by the geodesic dither presented in (\ref{gd}). Then, subject to Assumption \ref{as}, the averaged dynamical system of (\ref{kk}) is a perturbation of the gradient system (\ref{sc}) of the cost function $J$.
	\end{lemma}
	\begin{proof}
See Appendix \ref{A0}.
	\end{proof}
	
	The results of Lemma \ref{kir} imply that the state trajectories of the averaged dynamical system (\ref{av}) can be estimated by the state trajectories of  the scaled gradient system (\ref{sc}). In the case $a_{i}=a_{j},\hspace{.2cm}i,j=1,\cdots,n$ then $\dot{x}(t)=-\sum^{n}_{i=1}\frac{a^{2}_{i}}{2}\nabla_{\frac{\partial}{\partial x_{j}}}J(x(t))\frac{\partial}{\partial x_{i}}$ is identical to (\ref{ggs}).
	\newtheorem{remark}{Remark}
	\begin{remark}
	\label{rm}
	Note that the compactness of $M$ can be relaxed when the analysis is carried out in a local neighborhood of $x^{*}$ which is contained in a compact set. The existence of this compact set is guaranteed since manifolds are Housdorff spaces and Housdorff spaces are locally compact, see \cite{Lee4}, Proposition 4.27.
	
	\end{remark}
	We analyze the properties of the state trajectory of (\ref{kk}) based on the state trajectory of the average system (\ref{av}). Also stability properties of the gradient system $\dot{x}(t)=-\sum^{n}_{i=1}\frac{a^{2}_{i}}{2}\nabla_{\frac{\partial}{\partial x_{j}}}J(x(t))\frac{\partial}{\partial x_{i}}$ facilitate the closeness of solutions between the time varying dynamical systems (\ref{kk}) and the gradient system (\ref{sc}). The same results on Euclidean spaces are presented in \cite{sand}.  The following lemma gives the uniform local asymptotic stability of the gradient system obtained in the proof of Lemma \ref{kir}.
	
	\begin{lemma}
	\label{kkir2}
	Consider the gradient dynamical system  (\ref{sc}) on a compact connected $n$ dimensional Riemannian manifold $(M,g^{M})$. Then, subject to Assumption \ref{as}, the cost function $J$ is a Lyapunov function and  $x^{*} $  is locally asymptotically stable (see \cite{Kha}) for the gradient system (\ref{sc})  on $(M,g^{M})$ 
	\end{lemma}
	\begin{proof}
See Appendix \ref{A00}.
	\end{proof}

  The following theorem is the main result of this paper which gives a local convergence of the geodesic extremum seeking system  to a unique local minimum/maximum of the function $J$ on an $n$ dimensional compact Rimeannian manifold $(M,g^{M})$. The compactness assumption can be relaxed if the analysis is restricted to a local neighborhood of the optimizer $x^{*}$ which is contained in a compact set as per Remark \ref{rm}.
  \begin{theorem}
	\label{t1}
  Consider the geodesic extremum seeking system (\ref{kkir}) on a compact connected $n$ dimensional Riemannian manifold $(M,g^{M})$, where $\omega_{i}=\omega\bar{\omega}_{i}, i=1,\cdots,n$ satisfy Assumption \ref{as}. Assume $x^{*}\in M$ is a unique local optimizer of $J:M\rightarrow \mathds{R}_{\geq 0}$, where $J$ satisfies Assumption \ref{as}. Then for any neighborhood $U_{x^{*}}\subset M$ of $x^{*}$ on $M$, there exist sufficiently small parameters $a_{i}>0,\hspace{.2cm}i=1,\cdots,n$, sufficiently large  frequency $\omega$ and a neighborhood of $x^{*}$ denoted by $\hat{U}_{x^{*}}\subset M$, such that for any $x_{0}\in \hat{U}_{x^{*}}$, the state trajectory of the closed loop system (\ref{kk}) initiating from $x_{0}$ ultimately enters and remains in $U_{x^{*}}$.
  \end{theorem}
	\begin{proof}
	See Appendix \ref{A000}. 
	\end{proof}
	
  \begin{remark} The results of Theorem \ref{t1} give closeness of the state trajectory of the extremum seeking system (\ref{kk}) to the local optimizer $x^{*}$ since the state trajectory of (\ref{kk}) ultimately enters and remains in $U_{x^{*}}$. This is guaranteed by tuning the frequency $\omega$ and dither amplitudes $a_{i}>0$. 
	\end{remark}
	\begin{remark}
	The proof of Theorem \ref{t1} is based on closeness of solutions and averaging analysis for dynamical systems evolving on Riemannian manifolds and stability of perturbed systems.The results related to averaging for dynamical systems on Riemannian manifolds are presented in Appendix \ref{A1}. The stability results  for perturbed systems on Riemannian manifolds are presented in appendix \ref{A2}. 
 We employ the averaging techniques presented in Appendix \ref{A1} to analyze the closeness of solutions on Riemannian manifolds where the averaged dynamical system is not necessarily stable. However, one may show  that the state trajectory of the averaged system remains bounded in a neighborhood of $x^{*}$.
	\end{remark}
	
	The following lemma presents the state trajectory of a special combination of flows on $(M,g^{M})$. The proof of Theorem \ref{t1} is based on closeness of solutions of the state trajectory of extremum seeking system (\ref{kk}) and the trajectory $z(\cdot)$ constructed in the lemma below. 
	\begin{definition}
	\label{koskalak}
	Let $X,Y\in \mathfrak{X}(M)$ be smooth  vector fields on $M$, where it may be shown that $\Phi_{Y}(t,t_{0},.):M\rightarrow M$ is a local diffeomorphism (see \cite{mar1}). Let us denote  $T_{x}\Phi_{Y}^{(t,t_{0})^{-1}}$ as the pushforward of $\Phi_{Y}^{-1}(t,t_{0},.):M\rightarrow M$ at $x\in M$. Define the \textit{pull back} of $\Phi_{Y}^{(t,t_{0})}\doteq \Phi_{Y}(t,t_{0},\cdot)$ denoted by $\Phi_{Y}^{(t,t_{0})^{*}}$ as follows.
\EQ\label{pu} &&\Phi_{Y}^{(t,t_{0})^{*}}:\mathfrak{X}(M)\rightarrow \mathfrak{X}(M),\nnum\\&& \left(\Phi_{Y}^{(t,t_{0})^{*}}X\right)(x)\doteq T_{\Phi_{Y}(t,t_{0}, x)}\Phi_{Y}^{(t,t_{0})^{-1}}X(\Phi_{Y}(t,t_{0}, x)),\nnum\\&& X\in \mathfrak{X}(M), x\in M, t\in\mathds{R}.\EN
In $\mathds{R}^{n}$, for a diffeomorphism $\phi:\mathds{R}^{n}\rightarrow \mathds{R}^{n}$ and $X\in \mathfrak{X}(\mathds{R}^{n})$, we have 
\EQ (\phi^{*}X)(x)=\big(\frac{\partial \phi^{-1}}{\partial x}\circ X\circ \phi\big)(x)=\frac{\partial \phi^{-1}}{\partial x}(X(\phi(x))).\nnum\EN
	\end{definition}
	
	\begin{lemma}
	\label{ll11}
	Consider a $T$ periodic time varying dynamical system $\dot{x}=\epsilon f(x,t)$, where $f(x,t+T)=f(x,t)$, on an $n$ dimensional Riemannian manifold  $(M,g^{M})$. The averaged dynamical system is given by $\dot{x}=\epsilon\hat{f}(x)$, where $\hat{f}(x)=\frac{1}{T}\int^{T}_{0}f(x,s)ds$ and $\epsilon\in \mathds{R}_{\geq 0}$. Consider the combination of state flows $z(t)\doteq\Phi_{\epsilon Z}^{(1,0)}\circ \Phi_{\epsilon f}(t,t_{0},x_{0})\doteq \Phi_{\epsilon Z}(1,0,\Phi_{\epsilon f}(t,t_{0},x_{0}))\in M$ where $Z(t,x)\doteq\int^{t}_{0}\big(\hat{f}(x)-f(x,s)\big)ds$. Then $z(\cdot)$ satisfies 
\EQ \dot{z}(t)&=&\epsilon\Big[(\Phi^{-1})_{\epsilon Z}^{(1,0)^{*}} f+\int^{1}_{0}(\Phi^{-1})^{(1,s)^{*}}_{\epsilon Z}\big(\hat{f}-f\big)ds\Big]\circ z(t),\nnum\EN
where $(\Phi^{-1})^{(1,s)^{*}}_{\epsilon Z}$ is the pullback of the local diffeomorphism $(\Phi^{-1})^{(1,s)}_{\epsilon Z}\doteq \Phi^{-1}_{\epsilon Z}(1,s,\cdot)$.
	\end{lemma}
	\begin{proof}
	See Appendix \ref{A1}.
	\end{proof}
	
	The results of Lemma \ref{ll11} are employed to obtain the closeness of solutions for the dynamical systems (\ref{kk}) and its averaged system as per Lemma \ref{kir2}. The full proof of Theorem \ref{t1} is given in Appendix \ref{A000}.

	\section{Extremum seeking on Lie groups}
	\label{slie}
	The extremum seeking system (\ref{kkir}) is modified for optimization on Lie groups. In this case we employ the group structure of the ambient manifold and define the extremum seeking vector field along the exponential maps on Lie groups. This makes the computation of the geodesic dithers defined before particularly simple for matrix Lie groups.  The following lemma characterizes goedesics on Lie groups which admit bi-invariant Riemannian metrics.
	\begin{lemma}
	\label{lgg}
Assume $G$ is a Lie group which admits a bi-invariant Riemannian metric then for a left invariant vector field $X$ on $G$, $ g\star \exp(tX(e))$ is a geodesic emanating from $g\in G$.

\end{lemma}
\begin{proof}
The proof is a straightforward extension of Lemma \ref{lg}. To show that $\gamma(t)=g\star \exp(tX), \hspace{.2cm}X\in \mathcal{L}$ is a geodesic through $g\in G$ we have
			\EQ \nabla^{\mathds{G}}_{\dot{\gamma}(t)}\dot{\gamma}(t)&=&\nabla^{\mathds{G}}_{T_{\exp(tX)}g\frac{d\exp(tX)}{dt}}T_{\exp(tX)}g\frac{d\exp(tX)}{dt}\nnum\\&=&T_{ \exp(tX)}g\nabla^{\mathds{G}}_{\frac{d\exp(tX)}{dt}}\frac{d\exp(tX)}{dt}=0,\nnum\EN
			since $\exp(tX)$ is a geodesic by Lemma \ref{lg} and $\nabla^{\mathds{G}}_{\frac{d\exp(tX)}{dt}}\frac{d\exp(tX)}{dt}=0$
\end{proof}

 The geodesic dither (\ref{gd}) is given along $\exp$ on Lie groups by 
\EQ \label{ed}g_{p}(t)=g\star \exp\left(\sum^{n}_{i=1}a_{i}\sin(\omega_{i}t)\frac{\partial}{\partial g_{i}}\right), \EN
where $\frac{\partial}{\partial g_{i}}$ are the base elements of $\mathcal{L}$. The extremum seeking vector field on $G$ is then defined by
\EQ \label{kkg}&&f(g,t)\doteq - \sum^{n}_{i=1}a_{i}\sin(\omega_{i}t)J\left(g\star \exp \left(\sum^{n}_{j=1}a_{j}\sin(\omega_{j}t)\frac{\partial}{\partial g_{j}}\right)\right)T_{e}g\left(\frac{\partial}{\partial g_{i}}\right)\in T_{g}G,\nnum\\\hspace{1.2cm}&&i=1,\cdots,n,\EN
	where $T_{e}g\left(\frac{\partial}{\partial g_{i}}\right),\hspace{.2cm}i=1,\cdots,n$ are tangent vectors at $T_{g}G$. One may easily verify that $T_{e}(g_{1}\star g_{2})\left(\frac{\partial}{\partial g_{i}}\right)=T_{g_{2}}g_{1}\left(T_{g_{2}}\left(\frac{\partial}{\partial g_{i}}\right)\right)$ which shows that $T_{e}g\left(\frac{\partial}{\partial g_{i}}\right)$ is a left invariant vector field. Note that $f(g,t)$ is not necessarily left invariant since $J(g_{1}\star g_{2}\star \exp (\sum^{n}_{i=1}a_{i}\sin(\omega_{i}t)))=J( g_{2}\star \exp (\sum^{n}_{i=1}a_{i}\sin(\omega_{i}t)))$ is not true in general.  
	
	The following theorem gives the stability of the geodesic extremum seeking algorithm (\ref{kkg}) on compact Lie groups.
	\begin{theorem}
	\label{tt1}
  Consider the extremum seeking system (\ref{kkg}) on a compact connected $n$ dimensional Lie group $G$, where $\omega_{i}=\omega\bar{\omega}_{i},\hspace{.2cm}i=1,\cdots,n$ satisfy Assumption \ref{as}. Assume $g^{*}\in G$ is a unique local optimizer of $J:G\rightarrow \mathds{R}_{\geq 0}$, where $J$ satisfies Assumption \ref{as}. Then for any neighborhood $U_{g^{*}}\subset G$ of $G^{*}$ on $G$, there exist sufficiently small parameters $a_{i}>0, \hspace{.2cm}i=1,\cdots,n$,  sufficiently large $\omega$ and a neighborhood $\hat{U}_{g^{*}}\subset G$ of $g^{*}$ such that for any $g_{0}\in \hat{U}_{g^{*}}$, the state trajectory of (\ref{kkg}) initiating from $g_{0}$ ultimately enters and remains in $U_{g^{*}}$.
  \end{theorem}
	\begin{proof}
	Since $G$ is compact then it possesses a bi-invariant Riemannian metric. Hence, Lemma \ref{lgg} implies that $g\star \exp(\sum^{n}_{i=1}a_{i}\sin(\omega_{i}t)\frac{\partial}{\partial g_{i}}) $ is a geodesic through $g\in G$ and the proof is identical to the proof of Theorem \ref{t1}. 
	\end{proof}
	
	In the case that $G$ is not compact we employ the Taylor expansion of smooth functions on $G$, given in \cite{knap}. 
	\begin{lemma}[\hspace{-.01cm}\cite{knap}]
	Consider a left invariant vector field $X\in \mathfrak{X}(G)$ which is identified by $X(e)\in \mathcal{L}$. Then for a smooth function $J:G\rightarrow \mathds{R}$ we have 
	\EQ J(g\star \exp(X))=\sum^{m}_{j=1}\frac{1}{j!}(X^{j}J)(g)+\frac{1}{m!}\int^{1}_{0}(1-s)^{m}(X^{m+1}J)(g\star \exp(s X))ds,\nnum\EN
	where $X^{j}J(g)=\frac{d^{j}}{dt^{j}}J(g\star \exp(t X))\left|_{t=0}\right.$.
	\end{lemma}
	
	\begin{lemma}
	\label{kirak}
	Consider the extremum seeking algorithm in (\ref{kkg}) on a Lie group $G$ where the optimizing trajectory is perturbed by the exponential dither presented in (\ref{ed}). Then, subject to Assumption \ref{as}, the averaged dynamical system of (\ref{kkg}) is a perturbation of the first order variation $X_{a}(g)J$, where $X_{a}(g)$ is the left invariant vector field identified by $X_{a}(e)=\sum^{n}_{i=1}a^{2}_{i}\frac{\partial}{\partial g_{i}}$.
	\end{lemma}
	\begin{proof}
	Parallel to the proof of Lemma \ref{kir} we have 
	\EQ &&J(g\star \exp\sum^{n}_{i=1}a_{i}\sin(\omega_{i}t)\frac{\partial}{\partial g_{i}})=J(g)+(XJ)(g)+\cdots+\nnum\\&&\frac{1}{(m-1)!}\times\left(X^{m-1}J\right)(g)+\frac{1}{(m-1)!}\times\int^{1}_{0}(1-s)^{m-1}X^{m}\nnum\\&&J\left(g\star\exp s \sum^{n}_{i=1}a_{i}\sin(\omega_{i}t)\frac{\partial}{\partial g_{i}}\right)ds,\nnum\EN
	where $X(g,t)=\sum^{n}_{i=1}a_{i}\sin(\omega_{i}t)T_{e}g\frac{\partial}{\partial g_{i}}$. Following the proof of Lemma \ref{kir}, after averaging and changing  the time scale to $\tau=\omega t$ we have
	\EQ \label{ti1}\frac{dg}{d\tau}=&&-\frac{1}{\omega}\sum^{n}_{i=1}\frac{a^{2}_{i}}{2}\left(T_{e}g\frac{\partial}{\partial g_{i}}J\right)(g)T_{e}g\frac{\partial}{\partial g_{i}}+\frac{1}{\omega}\sum^{n}_{i=1}O((\max_{i\in1,\cdots,n}a_{i})^{4})T_{e}g\frac{\partial}{\partial g_{i}},\EN
	which completes the proof. 
	\end{proof}
	
	Note that since $G$ is not compact the perturbation vector field\\ $\frac{1}{\omega}\sum^{n}_{i=1}O((\max_{i\in1,\cdots,n}a_{i})^{4})T_{e}g\frac{\partial}{\partial g_{i}}$ is not uniformly bounded on $G$. However, at any point $g\in G$ its magnitude is of order $O((\max_{i\in1,\cdots,n}a_{i})^{4})$. It is shown in the proof of Theorem \ref{t7} in Appendix \ref{A2} that selecting the perturbation of an asymptotic stable system on a Riemannian manifold sufficiently small guarantees that the state trajectory of the perturbed system remains in a compact neighborhood of the equilibrium of the asymptotic stable system. In that case the  magnitude of perturbation vector field above can be uniformly bounded on a compact set containing the equilibrium. 
	
	\begin{lemma}
	\label{kakir2}
	Consider the gradient dynamical system \\$\dot{g}=-\sum^{n}_{i=1}\frac{a^{2}_{i}}{2}\left(T_{e}g\frac{\partial}{\partial g_{i}}J\right)(g)T_{e}g\frac{\partial}{\partial g_{i}},\hspace{.2cm}g\in G$ on an $n$ dimensional Lie group $(G,\star)$. Then subject to Assumption \ref{as}, $g^{*}$ is  is locally asymptotically stable on $(G,\star)$ for the gradient dynamical system. 
	\end{lemma}
	\begin{proof}
	The proof parallels the proof iof Lemma \ref{kkir2} since 
	\EQ \label{lily}\mathcal{L}_{-\sum^{n}_{i=1}\frac{a^{2}_{i}}{2}\left(T_{e}g\frac{\partial}{\partial g_{i}}J\right)(g)T_{e}g\frac{\partial}{\partial g_{i}}}J(g)=-\sum^{n}_{i=1}\frac{a^{2}_{i}}{2}\left(T_{e}g\frac{\partial}{\partial g_{i}}J\right)^{2}(g)\leq 0.\EN
	\end{proof}
	
	Note that by Assumption \ref{as}, $\left(X J\right)(g)=\lim_{t\rightarrow 0}\frac{J(g\star\exp(tX(e)))-J(g)}{t}=0$ only at the unique local optimal point $g^{*}$. 
	
	\begin{theorem}
	\label{tt11}
  Consider the extremum seeking system (\ref{kkg}) on a connected $n$ dimensional Lie group $G$, where $\omega_{i}=\omega\bar{\omega}_{i},\hspace{.2cm}i=1,\cdots,n$ satisfy Assumption \ref{as}. Assume $g^{*}\in G$ is a unique local optimizer of $J:G\rightarrow \mathds{R}_{\geq 0}$, where $J$ satisfies Assumption \ref{as}. Then for any neighborhood $U_{g^{*}}\subset G$ of $G^{*}$ on $G$, there exist sufficiently small parameters $a_{i},\hspace{.2cm}i=1,\cdots,n$, sufficiently large $\omega$ and a neighborhood $\hat{U}_{g^{*}}\subset G$ of $g^{*}$ such that for any $g_{0}\in \hat{U}_{g^{*}}$, the state trajectory of (\ref{kkg}) initiating from $g_{0}$ ultimately enters and remains in $U_{g^{*}}$.
  \end{theorem}
	
	\begin{proof}
	See Appendix \ref{AK}.
	\end{proof}

	\section{Example on $SO(3)$}
	\label{s4}
	In this section we give a conceptual example  for  orientation control which is modeled by elements of $SO(3)$ which is a compact Lie group. 

We recall that $SO(3)$ is the rotation group in $\mathds{R}^{3}$ given by 
\EQ SO(3)= \big \{g\in GL(3)| \quad g.g^{T}=I,\hspace{.1cm} det(g)=1\big\},\nnum\EN 
where $GL(n)$ is the set of nonsingular $n\times n$ matrices.
The Lie algebra of $SO(3)$ which is denoted by $so(3)$ is given by (see \cite{Varad}) $so(3)=\big\{     X\in M(3)|\quad X+X^{T}=0\big\},$
where $M(n)$ is the space of all $n\times n$ matrices. The Lie group operation $\star$ is given by the matrix multiplication and consequently $T_{g_{1}}g_{2}$ is also given by the matrix multiplication $g_{2}X,\hspace{.2cm} X\in T_{g_{1}}G$.

A left invariant dynamical system on $SO(3)$ is given by 
\EQ \label{dd}\dot{g}(t)=gX,\quad g(0)=g_{0},\hspace{.2cm} X\in so(3).\EN
The Lie algebra bilinear operator is defined as the commuter of matrices, i.e. 
\EQ [X,Y]=XY-YX,\quad X,Y\in so(3).\nnum\EN 

Equation (\ref{dd})  above is written as
\EQ &&\left( \begin{array}{ll} \dot{g_{11}} \quad \dot{g_{12}} \quad \dot{g_{13}}\\\dot{g_{21}}\quad \dot{g_{22}}\quad \dot{g_{23}}\\\dot{g_{31}}\quad\dot{g_{32}}\quad \dot{g_{33}}
        \end{array}\right)=\left( \begin{array}{ll} g_{11} \quad g_{12} \quad g_{13}\\g_{21}\quad g_{22}\quad g_{23}\\g_{31}\quad g_{32}\quad g_{33}
        \end{array}\right)\left( \begin{array}{ll} 0 \quad \hspace{.5cm}X_{1}(t) \quad X_{3}(t)\\-X_{1}(t)\quad 0\quad\hspace{.3cm} X_{2}(t)\\-X_{3}(t)\quad \hspace{-.2cm}-X_{2}(t)\quad 0
        \end{array}\right).\nnum\EN

The optimization is performed for the cost function $J:SO(3)\rightarrow \mathds{R}$ defined by
\EQ J(g)=\frac{1}{2}tr((g-g^{*})^{T}(g-g^{*})),\nnum\EN
where $g^{*}$ is the optimal orientation matrix in $SO(3)$. We assume $g^{*}=I_{3\times 3}$, hence, $J=3-tr(g)$.
The Lie algebra $so(3)$ is spanned by $\frac{\partial}{\partial g_{1}}=\left( \begin{array}{ll} 0 \quad\hspace{.3cm} 1 \quad 0\\-1\quad 0\quad 0\\0\quad\hspace{.3cm} 0\quad 0
        \end{array}\right), \frac{\partial}{\partial g_{2}}=\left( \begin{array}{ll} 0 \quad\hspace{.3cm} 0 \quad 0\\0\quad\hspace{.3cm} 0\quad 1\\0\quad -1\quad 0
        \end{array}\right) \mbox{and}\hspace{.2cm} \frac{\partial}{\partial g_{3}}=\left( \begin{array}{ll} 0 \quad\hspace{.3cm} 0 \quad 1\\0\quad\hspace{.3cm} 0\quad 0\\-1\quad 0\quad 0
        \end{array}\right)$. For this example the dither vector $X(e)$ at the Lie algebra $so(3)$ is given by
				\EQ X(e)&&=\sum^{3}_{i=1}a_{i}\sin(\omega_{i}t)\frac{\partial}{\partial g_{i}}=\left( \begin{array}{ll} \hspace{1cm}0 \quad\hspace{.3cm} a_{1}\sin(\omega_{1}t) \quad a_{3}\sin(\omega_{3}t)\\-a_{1}\sin(\omega_{1}t)\quad \hspace{.5cm}0\quad  \hspace{.5cm}a_{2}\sin(\omega_{2}t)\\-a_{3}\sin(\omega_{3}t)\quad\hspace{-.3cm} -a_{2}\sin(\omega_{2}t)\quad 0
        \end{array}\right),\nnum\EN
				hence, the dither vector field is given by
				\EQ X(g)=g\cdot \left( \begin{array}{ll} \hspace{1cm}0 \quad\hspace{.3cm} a_{1}\sin(\omega_{1}t) \quad a_{3}\sin(\omega_{3}t)\\-a_{1}\sin(\omega_{1}t)\quad \hspace{.5cm}0\quad  \hspace{.5cm}a_{2}\sin(\omega_{2}t)\\-a_{3}\sin(\omega_{3}t)\quad\hspace{-.3cm} -a_{2}\sin(\omega_{2}t)\quad 0
        \end{array}\right),\nnum\EN
				where $g\in SO(3)$.

The extremum seeking vector field given in (\ref{kkg}) is presented on $SO(3)$  by 
\EQ\label{ll} &&\hspace{-1cm}f(g,t)\doteq -\sum^{3}_{i=1}a_{i}\sin(\omega_{i}t)J(g\exp\sum^{3}_{j=1}a_{j}\sin(\omega_{j}t)\frac{\partial}{\partial g_{j}})g\frac{\partial}{\partial g_{i}},\EN
where by Lemma \ref{lg} the exponential map on $SO(3)$ is a geodesic ($SO(3)$ is compact). 

Note that in (\ref{ll})  $g\frac{\partial}{\partial g_{i}}\in T_{g}SO(3)$.
  The optimizing trajectory $g(\cdot)$ is a solution of the time dependent differential equation 
  \EQ \dot{g}(t)=f(g,t)\in T_{g}SO(3). \nnum\EN
	The algorithms initiates from the initial orientation at $g_{0}=\left( \begin{array}{ll}\cos(\frac{\pi}{4}) \quad\hspace{-.3cm} -\sin(\frac{\pi}{4}) \quad 0\\\sin(\frac{\pi}{4})\quad \cos(\frac{\pi}{4})\quad \hspace{.05cm} 0\\0\quad\hspace{1cm}0\quad\hspace{.7cm} 1
        \end{array}\right)\in SO(3)$. The amplitudes and frequencies are set at $a_{1}=a_{2}=a_{3}=.1$ and $\omega_{1}=2,\omega_{2}=4.1,\omega_{3}=6.2$. Figure \ref{18}  shows the convergence of the cost function and the state trajectory $g(t)=\left( \begin{array}{ll}g_{11}(2) \quad g_{12}(t) \quad g_{13}(t)\\g_{21}(t)\quad g_{22}(t)\quad g_{23}(t)\\ g_{31}(t)\quad g_{32}(t)\quad g_{33}(t)\end{array}\right)\in SO(3)$.

\begin{figure}
	\begin{minipage}[b]{0.45\linewidth}
	
\hspace*{0cm}\includegraphics[scale=.3]{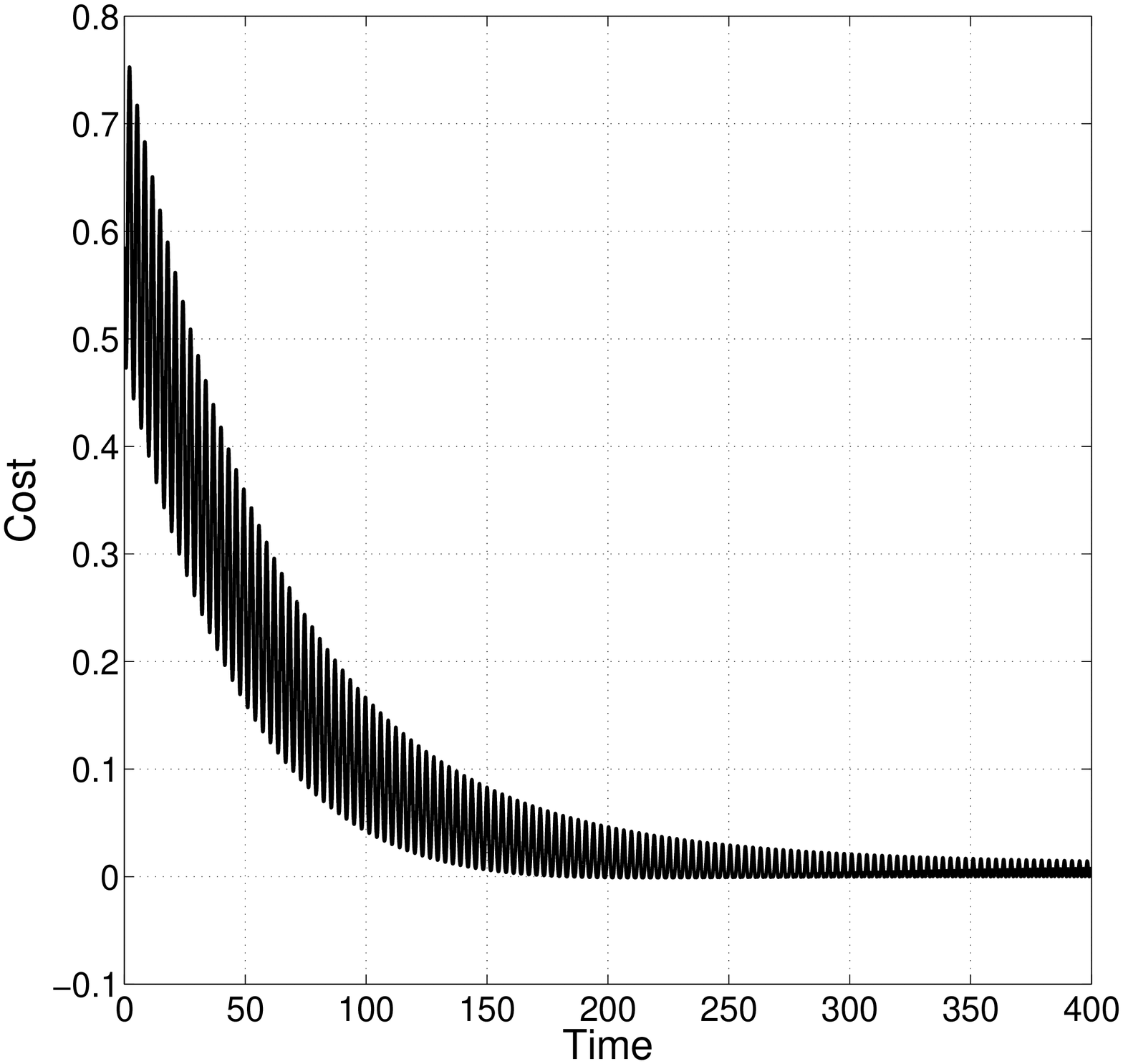}

\end{minipage}
\begin{minipage}[b]{0.45\linewidth}
	
\hspace*{0.2cm}\includegraphics[scale=.3]{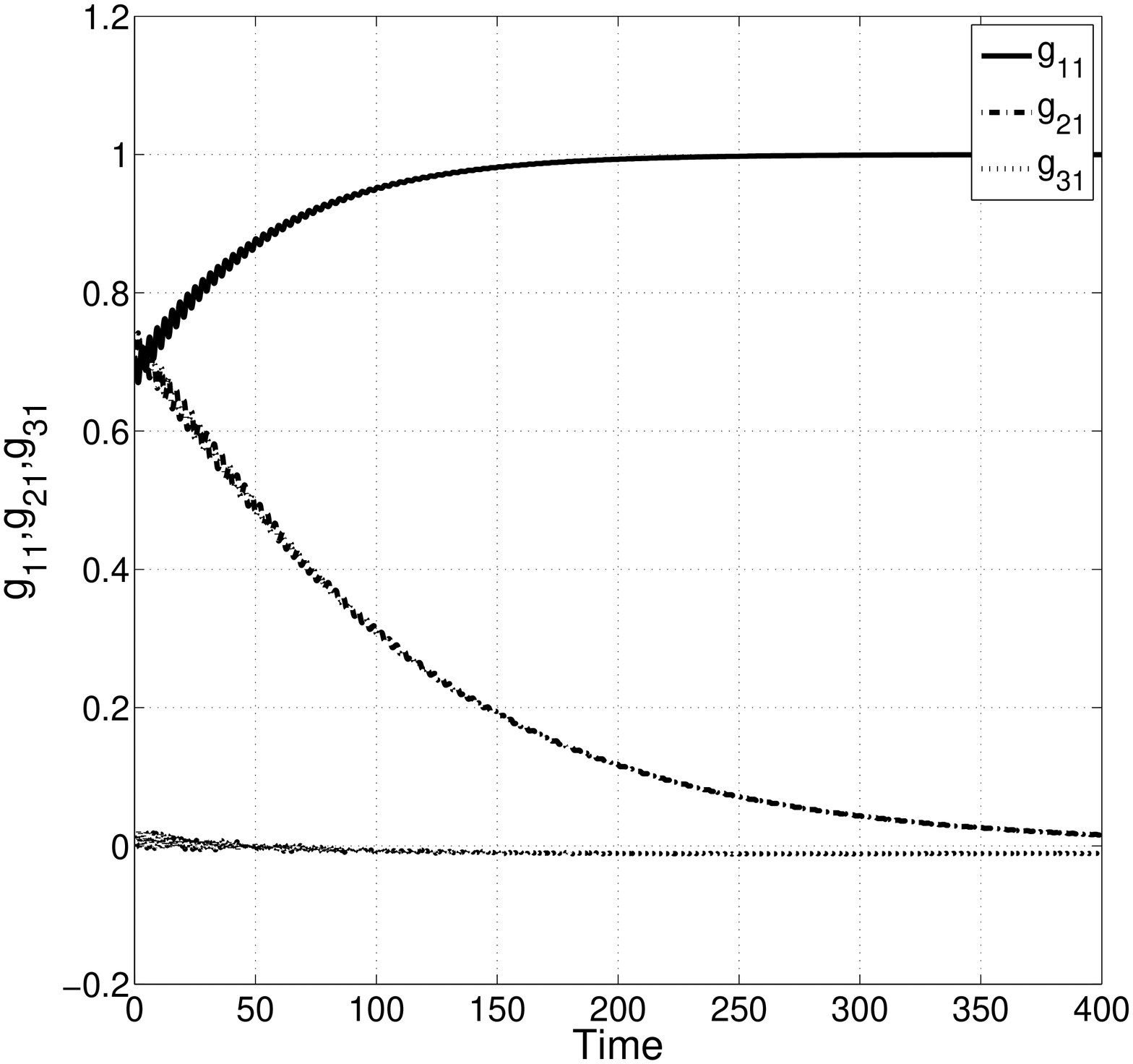}

\end{minipage}
\begin{minipage}[b]{0.45\linewidth}
	
\hspace*{0cm}\includegraphics[scale=.3]{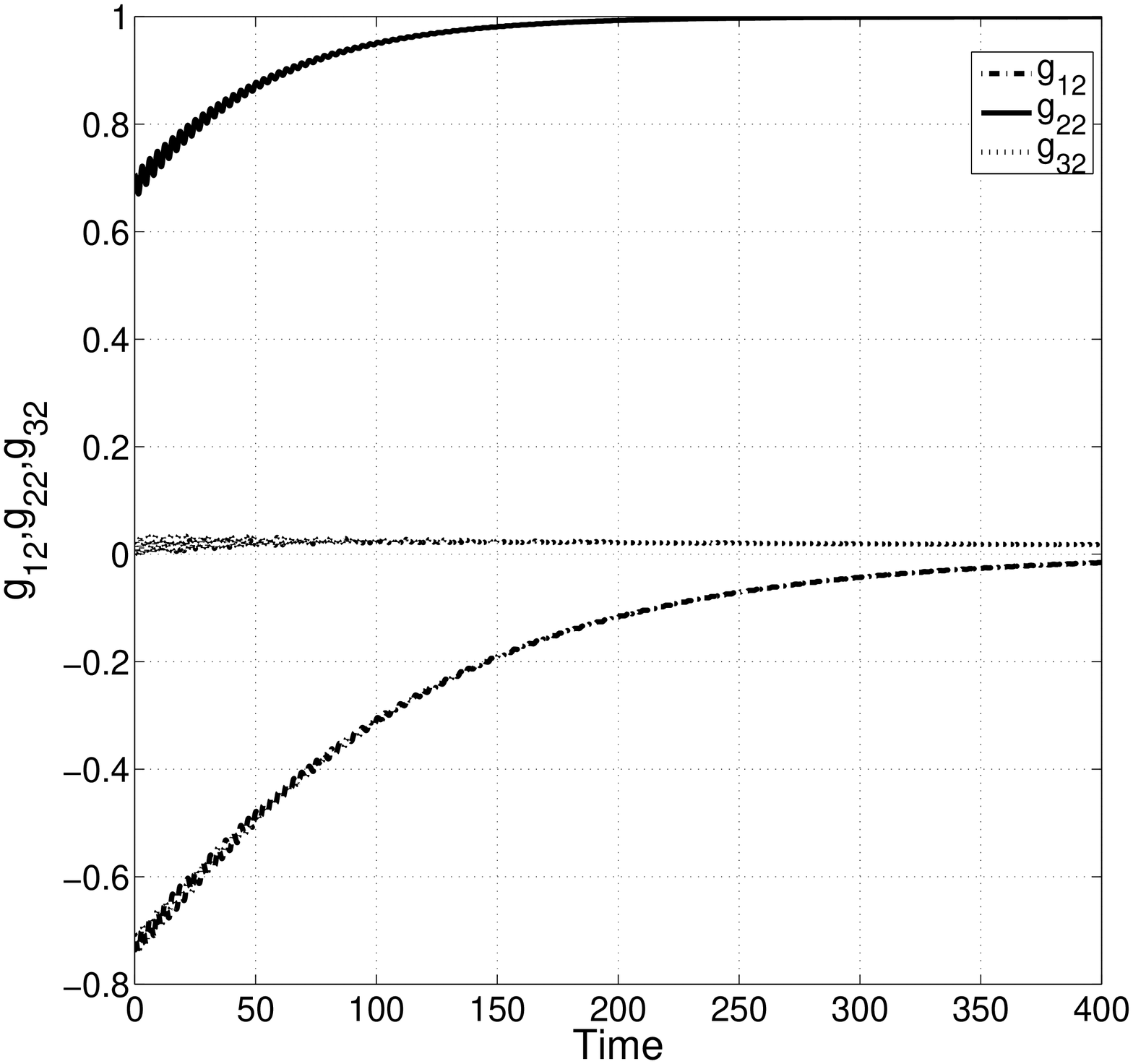}

\end{minipage}
\begin{minipage}[b]{0.45\linewidth}
	
\hspace*{1.5cm}\includegraphics[scale=.3]{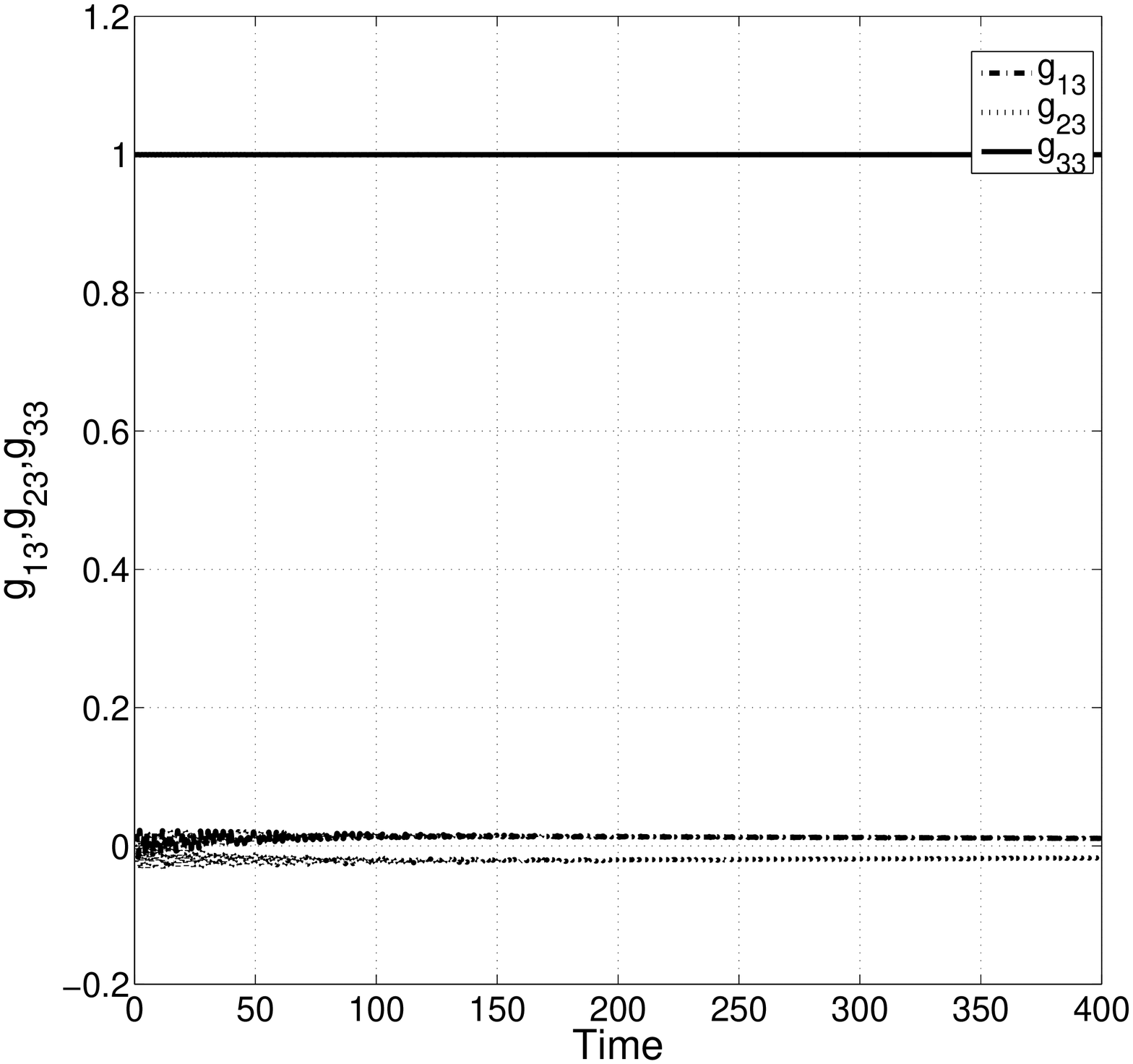}

\end{minipage}
\label{18}
\caption{Cost and state convergence on $SO(3)$. }
\end{figure}
   
As is obvious, the optimal solution for the optimization problem is $g^{*}=I_{3\times 3}$. The algorithm converges to $g=\left( \begin{array}{ll}.999 \quad-0.0159     \quad 0.0110\\0.0161   \quad  0.9998   \quad -0.0174\\-0.0108    \quad0.0175   \quad 0.9998
        \end{array}\right)\in SO(3)$.

\section{Example on $SE(3)$}
\label{s5}
	In this section we give another conceptual example  for an orientation control on $SE(3)$ which is not compact. 

As is known, $SE(3)$ is the space of rotation and translation which is used for robotic modeling. We have
\EQ SE(3)=\left\{\left( \begin{array}{ll} g_{SO(3)} \hspace{.3cm} g_{\mathds{R}} \\ 0_{1\times 3}\hspace{.6cm}  1
        \end{array}\right)\in \mathds{R}^{4\times 4}| \quad g_{SO(3)}\in SO(3), g_{\mathds{R}}\in \mathds{R}^{3\times 1}\right\},\nnum\EN
				where $g_{SO(3)}$ models the rotation and $g_{\mathds{R}}$ models the translation in $\mathds{R}^{3}$.
The Lie algebra of $SE(3)$ which is denoted by $se(3)$ is given by (see \cite{Varad})
 \EQ se(3)=\left\{  \left( \begin{array}{ll} S \hspace{.7cm} v \\ 0_{1\times 3}\hspace{.2cm}  0
        \end{array}\right)\in \mathds{R}^{4\times 4}|\quad S\in so(3), v\in \mathds{R}^{3}\right\},\nnum\EN

Let us consider the  cost function $J:SE(3)\rightarrow \mathds{R}$ as 
\EQ J(g)&=&\frac{1}{2}tr((g_{SO(3)}-g^{*}_{SO(3)})^{T}(g_{SO(3)}-g^{*}_{SO(3)}))+\frac{1}{2}||g_{\mathds{R}}-r^{*}||_{\mathds{R}^{3}}^{2},\nnum\EN

where $g^{*}_{SO(3)}$ is the optimal orientation matrix in $SO(3)$ and $r^{*}$ is the optimal distance from the origin in $\mathds{R}^{3}$. As is obvious the optimal solution for the optimization problem above is $\left(\begin{array}{ll} g^{*}_{SO(3)} \hspace{.3cm} r^{*} \\ 0_{1\times 3}\hspace{.5cm}  1\end{array}\right)\in SE(3)$. The cost function above minimizes the distance from the orientation $g^{*}_{SO(3)}$ and distance from $r^{*}$. Without loss of generality, we assume $g^{*}_{SO(3)}=I_{3\times 3}\in SO(3)$ and $r^{*}=(0,0,0)\in \mathds{R}^{3}$.

The Lie algebra $se(3)$ is spanned by\\ $\frac{\partial}{\partial g_{1}}=\left( \begin{array}{ll} 0 \quad 1 \quad 0 \quad 0\\\hspace{-.25cm}-1\quad 0\quad 0\quad 0\\0\quad 0\quad 0\quad 0\\0\quad 0\quad 0\quad0
        \end{array}\right), \frac{\partial}{\partial g_{2}}=\left( \begin{array}{ll} 0 \quad 0 \quad 1 \quad 0\\0\quad 0\quad 0\quad 0\\\hspace{-.25cm}-1\quad 0\quad 0\quad 0\\0\quad 0\quad 0\quad0
        \end{array}\right),\frac{\partial}{\partial g_{3}}=\left( \begin{array}{ll} 0 \quad 0 \quad 0 \quad 0\\0\quad 0\quad 1\quad 0\\0\quad \hspace{-.3cm}-1\quad 0\quad 0\\0\quad 0\quad 0\quad0
        \end{array}\right),\frac{\partial}{\partial g_{4}}=\left( \begin{array}{ll} 0 \quad 0 \quad 0 \quad 1\\0\quad 0\quad 0\quad 0\\0\quad 0\quad 0\quad 0\\0\quad 0\quad 0\quad0
        \end{array}\right), \frac{\partial}{\partial g_{5}}=\left( \begin{array}{ll} 0 \quad 0 \quad 0 \quad 0\\0\quad 0\quad 0\quad 1\\0\quad 0\quad 0\quad 0\\0\quad 0\quad 0\quad0
        \end{array}\right) \mbox{and}\hspace{.2cm} \frac{\partial}{\partial g_{6}}=\left( \begin{array}{ll} 0 \quad 0 \quad 0 \quad 0\\0\quad 0\quad 0\quad 0\\0\quad 0\quad 0\quad 1\\0\quad 0\quad 0\quad0
        \end{array}\right)$. For this example the dither vector $X(e)$ at the Lie algebra $se(3)$ is given by
				\EQ X(e)&&=\sum^{6}_{i=1}a_{i}\sin(\omega_{i}t)\frac{\partial}{\partial g_{i}}=\left( \begin{array}{ll} \hspace{1cm}0 \quad\hspace{.3cm} a_{1}\sin(\omega_{1}t) \quad a_{3}\sin(\omega_{3}t)\quad a_{4}\sin(\omega_{4}t)\\-a_{1}\sin(\omega_{1}t)\quad \hspace{.5cm}0\quad  \hspace{.5cm}a_{2}\sin(\omega_{2}t)\quad a_{5}\sin(\omega_{5}t)\\-a_{3}\sin(\omega_{3}t)\quad\hspace{-.3cm} -a_{2}\sin(\omega_{2}t)\quad 0\quad\hspace{.9cm} a_{6}\sin(\omega_{6}t)\\\hspace{1cm}0\quad\hspace{1.5cm} 0\quad\hspace{1cm} 0\quad \hspace{1.5cm}0
        \end{array}\right),\nnum\\\EN
				hence, the dither vector field is given by $X(g)=g\cdot X(e)$, where $g\in SE(3)$.

Similar to the example on $SO(3)$, the extremum seeking vector field on $SE(3)$ is given by the following vector field
\EQ &&\hspace{-1cm}f(g,t)\doteq -\sum^{6}_{i=1}a_{i}\sin(\omega_{i}t)J(g\exp\sum^{6}_{j=1}a_{j}\sin(\omega_{j}t)\frac{\partial}{\partial g_{j}})g\frac{\partial}{\partial g_{i}},\nnum\EN
where $\exp$ is the exponential operator defined on $SE(3)$. On $SE(3)$ the exponential map does not coincide with geodesics since $SE(3)$ does not admit a bi-invariant Riemannian metric, see \cite{pen}. Hence, the results of Theorem \ref{tt11} grantees the local convergence of the algorithm. 

 In this case, the $\exp$ operator is not the same as the $\exp$ operator on $SO(3)$. For a tangent vector $\left( \begin{array}{ll} S \hspace{.7cm} v \\ 0_{1\times 3}\hspace{.2cm}  0\end{array}\right)\in se(3)$, where $S=\left( \begin{array}{ll}\hspace{.25cm} 0 \quad a \quad b \\-a\quad 0\quad c\\-b\quad \hspace{-.25cm}-c\quad 0\end{array}\right)$, we have $\exp(\left( \begin{array}{ll} S \hspace{.7cm} v \\ 0_{1\times 3}\hspace{.2cm}  0\end{array}\right))=\left( \begin{array}{ll} \exp(S) \hspace{.1cm} Av \\ 0_{1\times 3}\hspace{.6cm}  1\end{array}\right)$,
where $A=I_{3\times 3}+\frac{(1-\cos(\theta))}{\theta^{2}}S+\frac{(\theta-\sin(\theta))}{\theta^{3}}S^{2}$, and $\theta=\sqrt{a^{2}+b^{2}+c^{2}}$. In the case that $\theta=0$, we have $\exp(\left( \begin{array}{ll} S \hspace{.7cm} v \\ 0_{1\times 3}\hspace{.2cm}  0\end{array}\right))=\left( \begin{array}{ll} \exp(S) \hspace{.2cm} v \\ 0_{1\times 3}\hspace{.6cm}  1\end{array}\right)$.

  The optimizing trajectory $g(\cdot)$ is a solution of the time dependent differential equation 
  \EQ \dot{g}(t)=f(g,t)\in T_{g}SE(3).\nnum \EN
	The algorithm initiates from the initial orientation at \\$g_{SO(3)}(0)=\left( \begin{array}{ll}\cos(\frac{\pi}{4}) \quad\hspace{-.3cm} -\sin(\frac{\pi}{4}) \quad 0\\\sin(\frac{\pi}{4})\quad \cos(\frac{\pi}{4})\quad \hspace{.05cm} 0\\0\quad\hspace{1cm}0\quad\hspace{.7cm} 1
        \end{array}\right)\in SO(3)$ and $g_{\mathds{R}}=(1,-1,2)\in \mathds{R}^{3}$. The amplitudes and frequencies are set at $a_{1}=\cdots=a_{6}=.1$ and $\omega_{i}=2i+\epsilon_{i},\hspace{.2cm}i=1,\cdots,6$.

				\begin{figure}
\begin{minipage}[b]{0.45\linewidth}
	
\hspace*{0cm}\includegraphics[scale=.3]{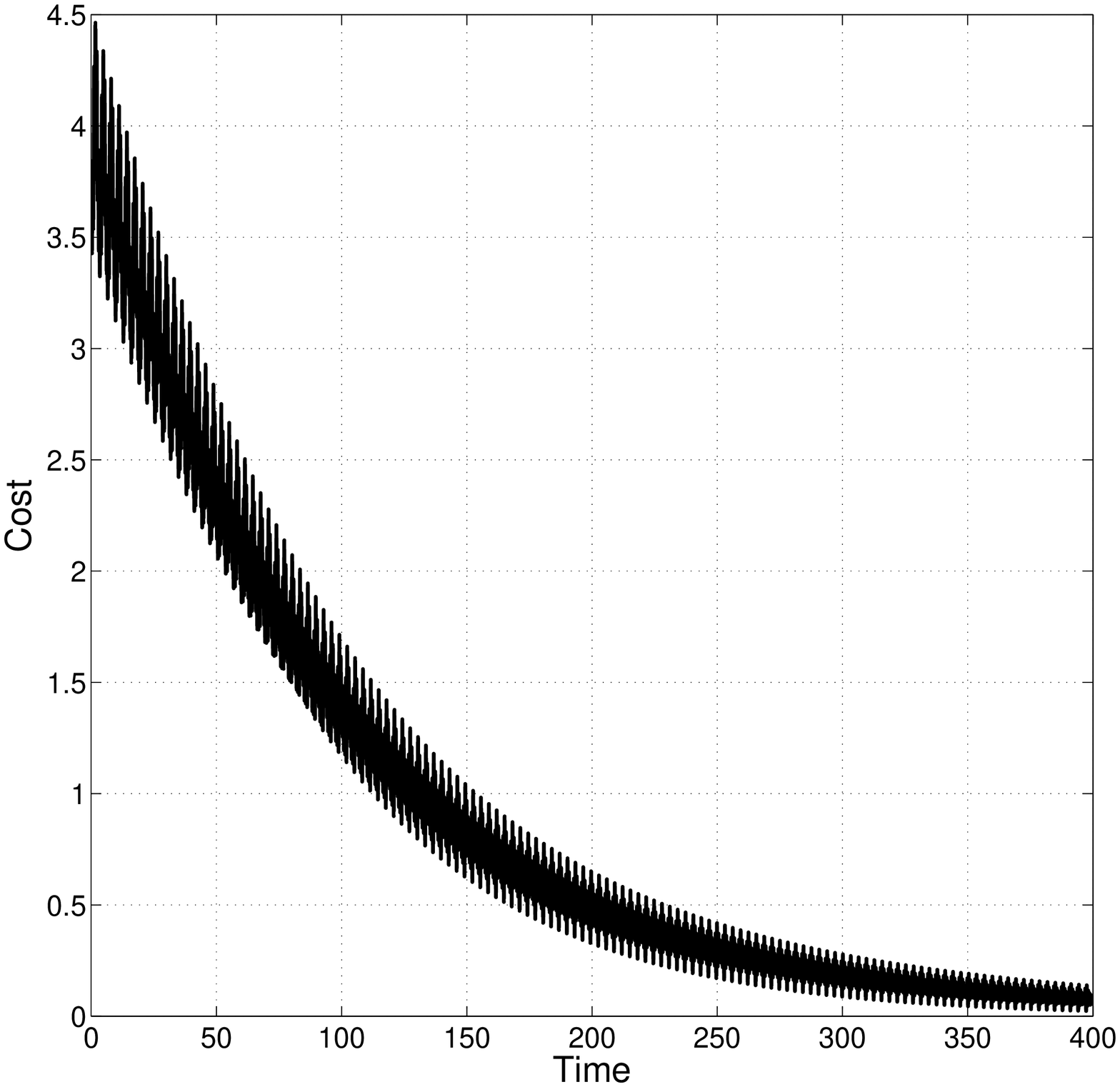}

\end{minipage}
\begin{minipage}[b]{0.45\linewidth}
	
\hspace*{1.5cm}\includegraphics[scale=.3]{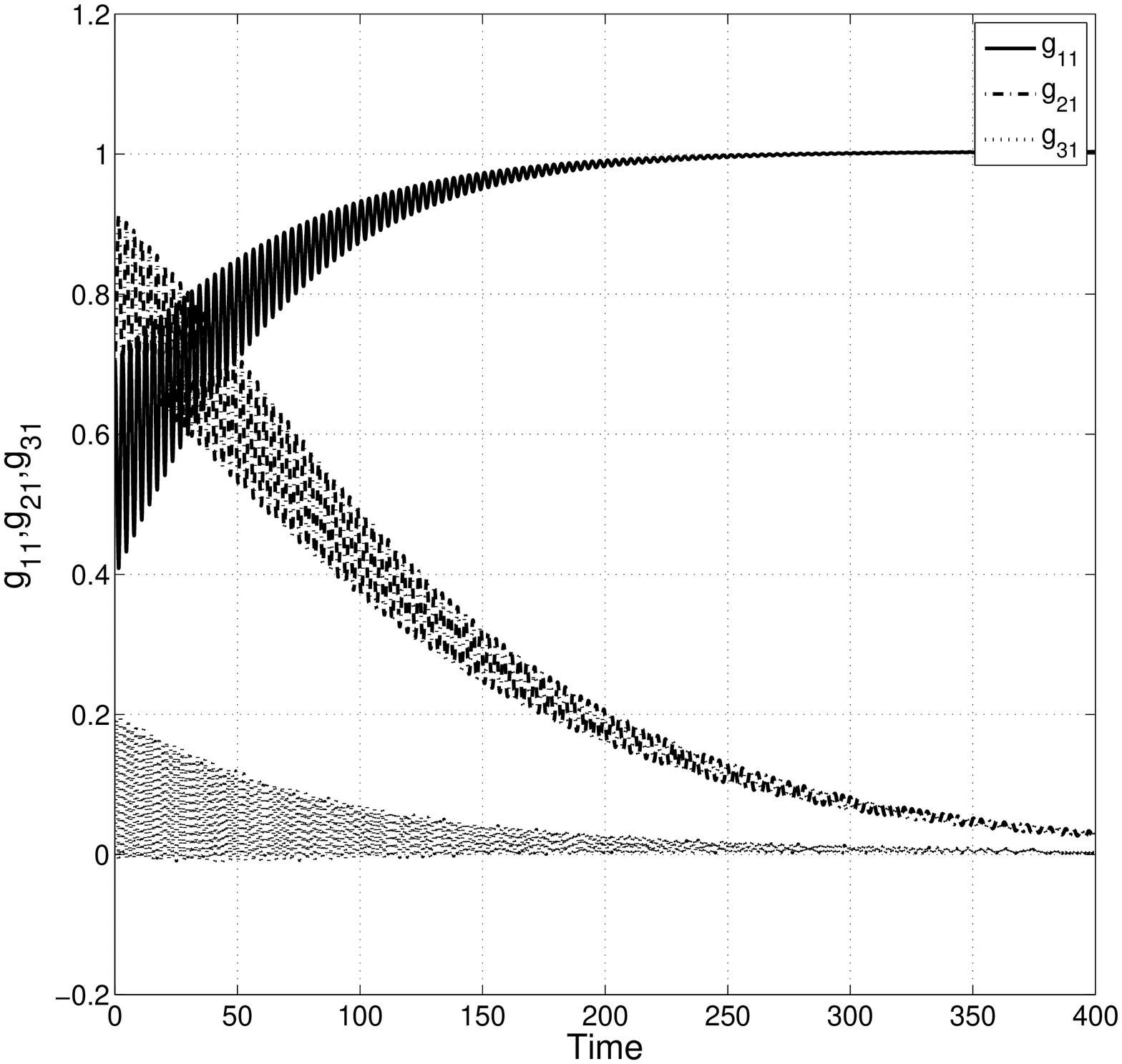}

\end{minipage}

\begin{minipage}[b]{0.45\linewidth}
	
\hspace*{4cm}\includegraphics[scale=.3]{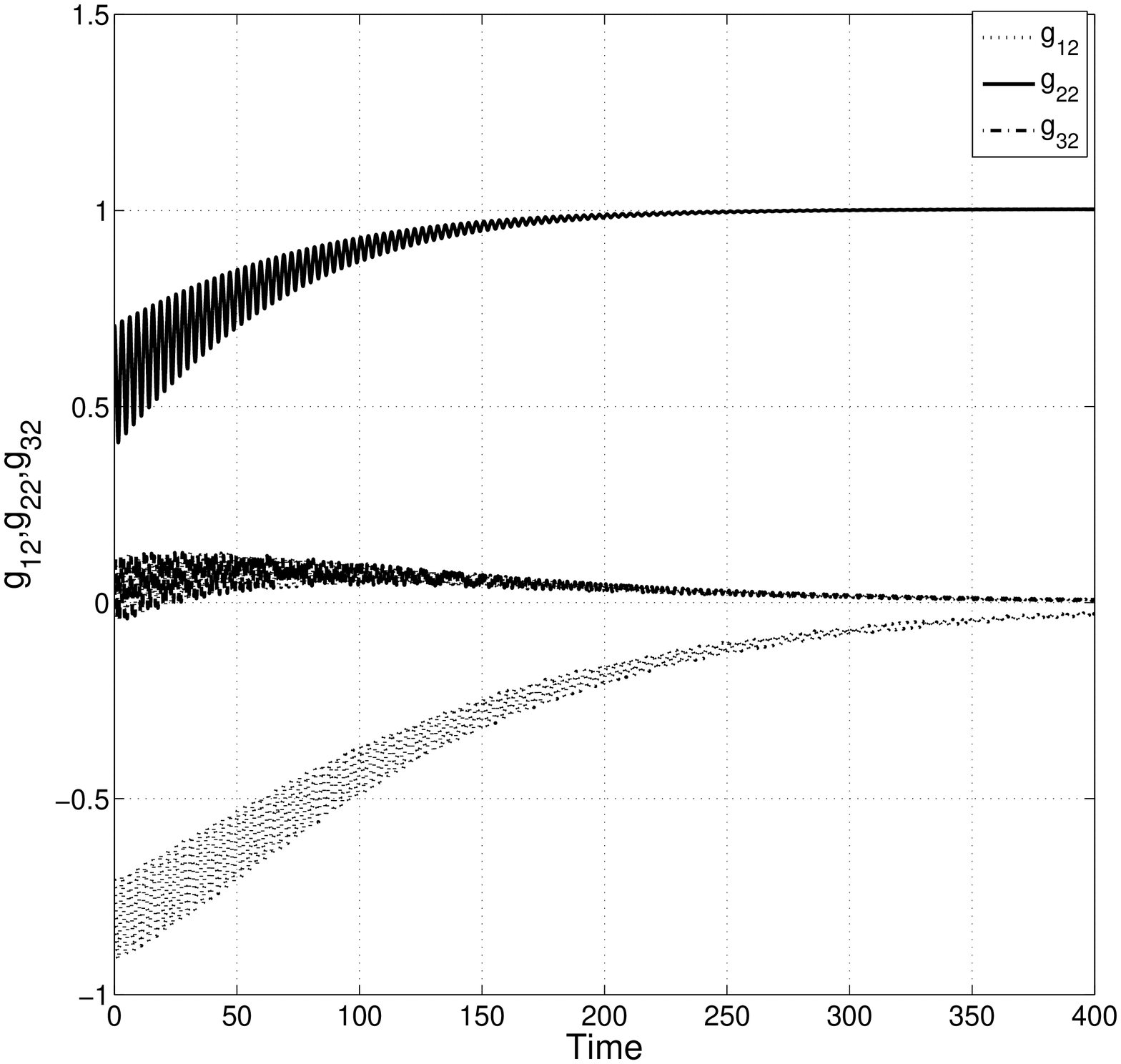}

\end{minipage}
 
 \caption{Cost and state convergence on $SE(3)$. }
     \label{33}
  \end{figure}
				
		\begin{figure}
\begin{minipage}[b]{0.45\linewidth}
	
\hspace*{0cm}\includegraphics[scale=.3]{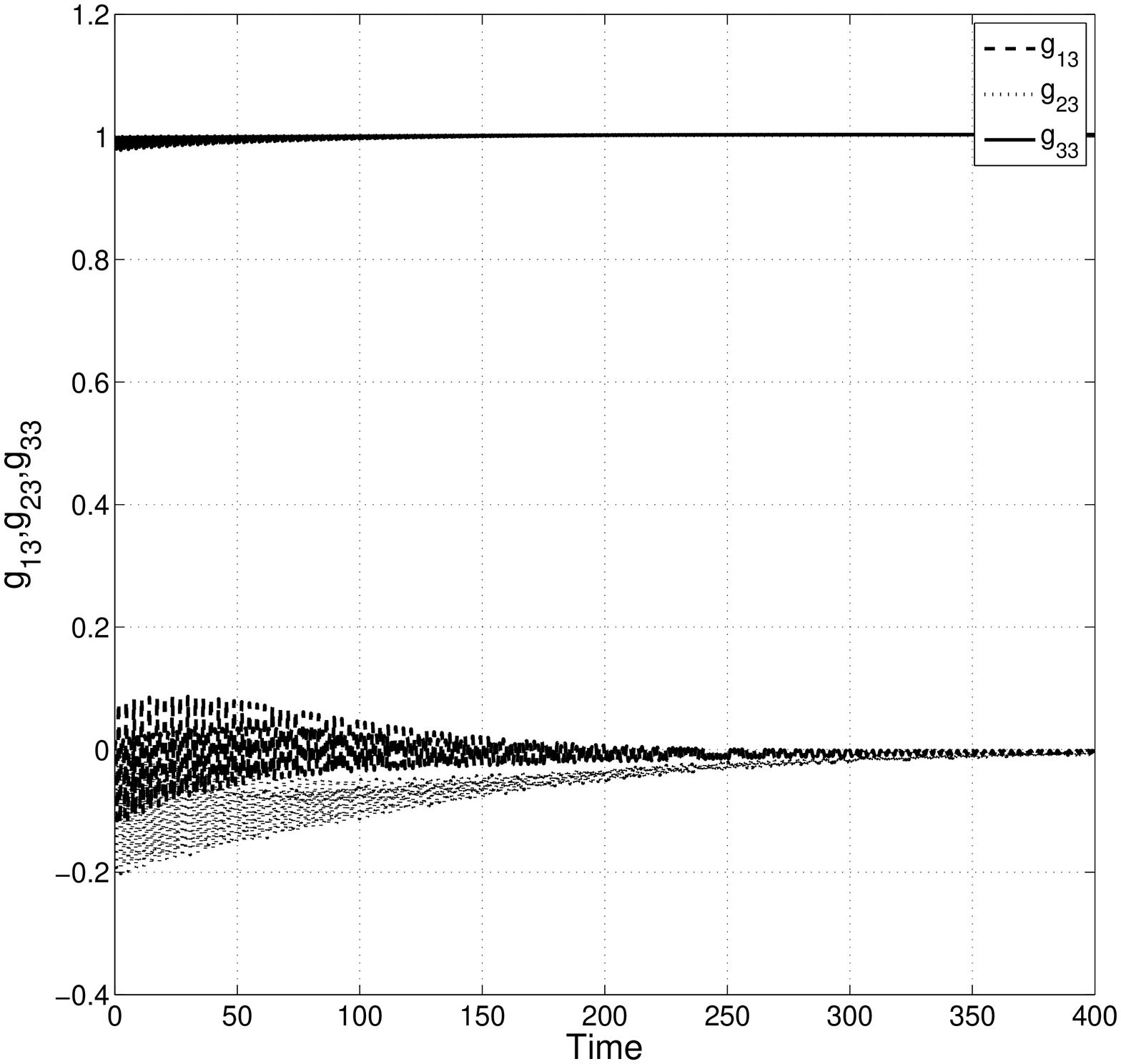}

\end{minipage}
\begin{minipage}[b]{0.45\linewidth}
	
\hspace*{1.5cm}\includegraphics[scale=.3]{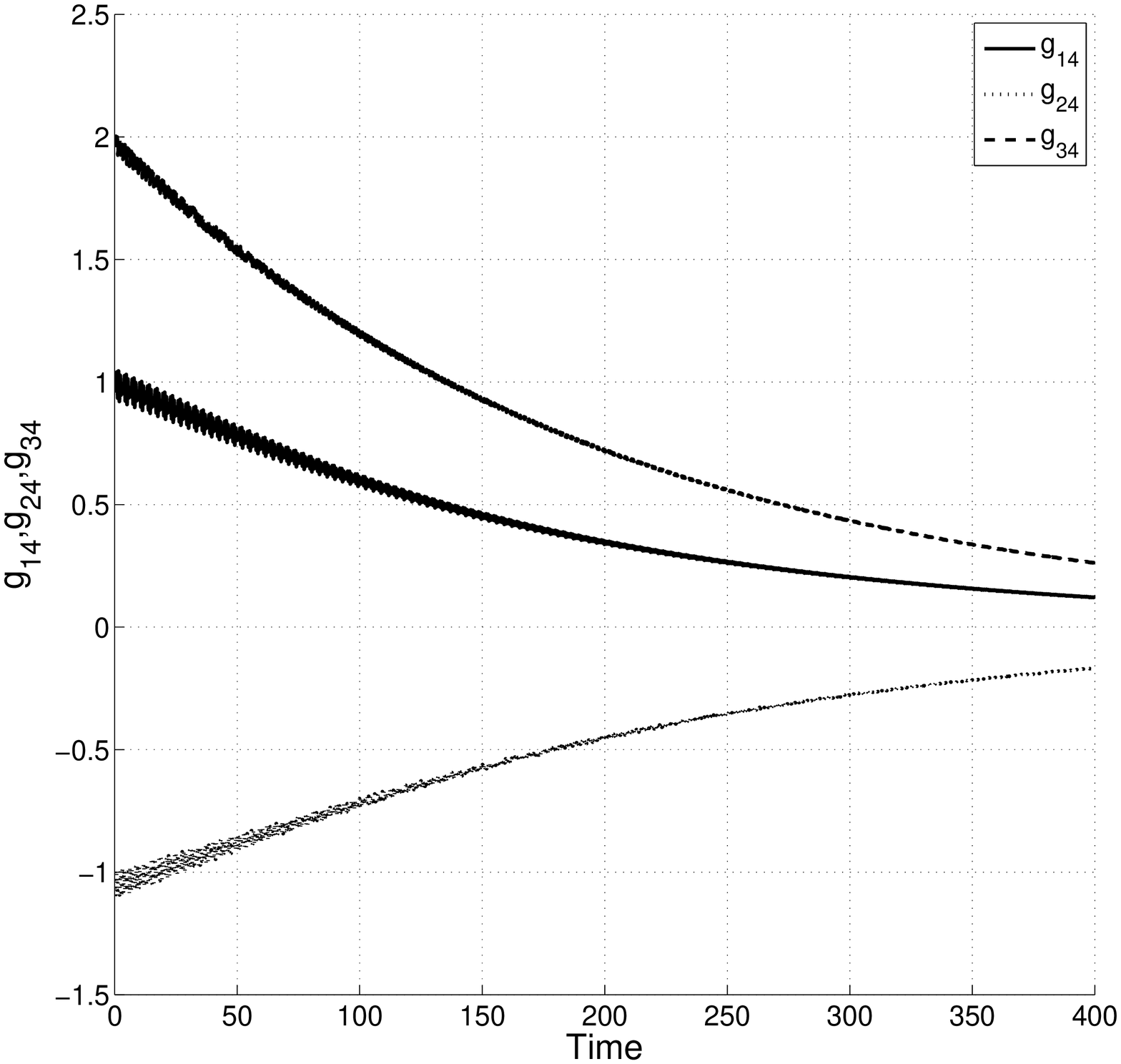}

\end{minipage}

\begin{minipage}[b]{0.45\linewidth}
	
\hspace*{5cm}\includegraphics[scale=.3]{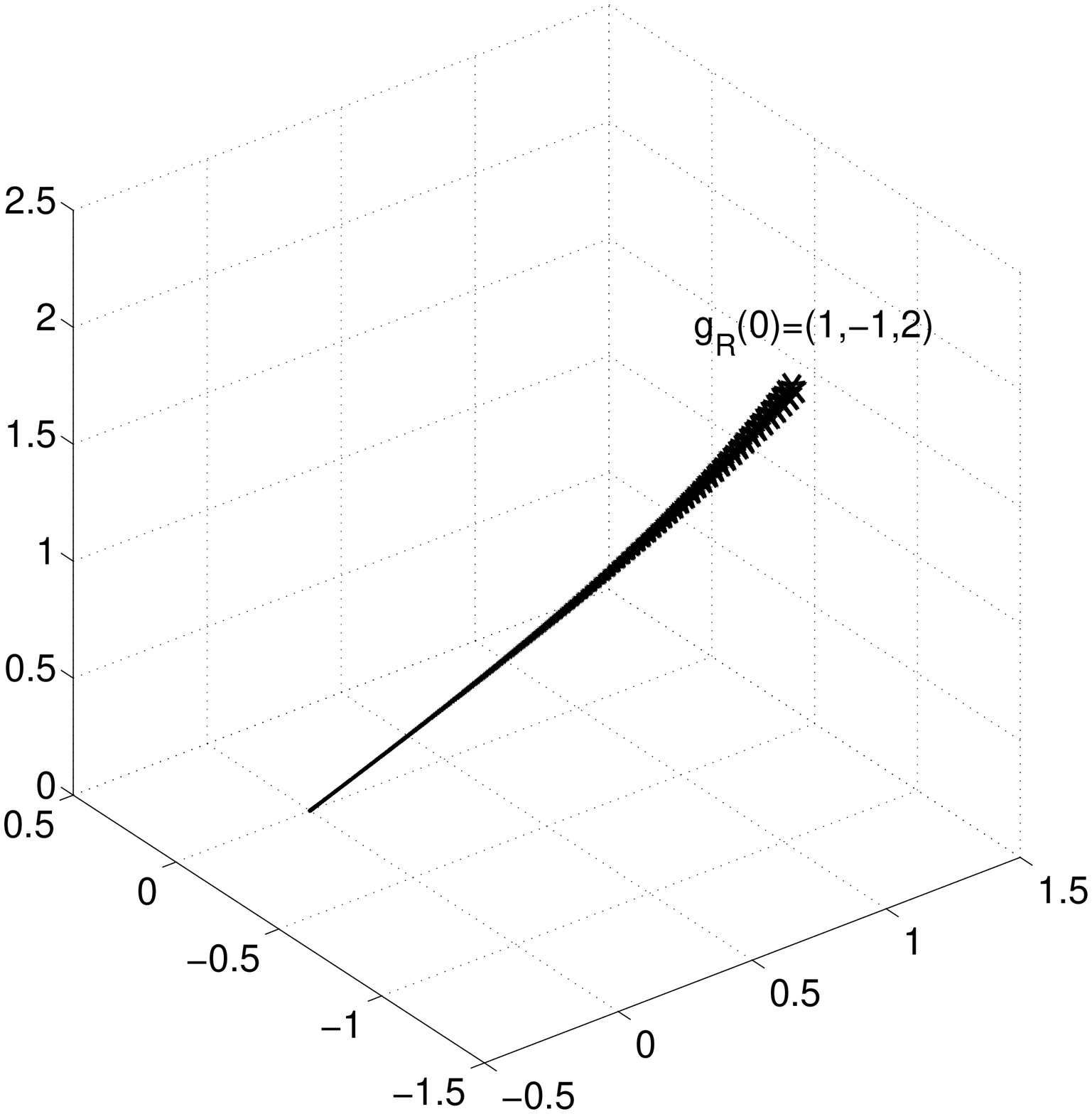}

\end{minipage}
 \caption{ State convergence on $SE(3)$. }
     \label{34}
     
  \end{figure}
	The results for the convergence of the cost function and the state trajectory $g(t)=\left( \begin{array}{ll}g_{11}(t) \quad g_{12}(t) \quad g_{13}(t)\quad g_{14}(t)\\g_{21}(t)\quad g_{22}(t)\quad g_{23}(t)\quad g_{24}(t)\\ g_{31}(t)\quad g_{32}(t)\quad g_{33}(t)\quad g_{34}(t)\\0\hspace{1cm}\quad 0\hspace{.7cm}\quad 0\hspace{.6cm}\quad 1\end{array}\right)\in SE(3)$ are shown in Figures \ref{33} and \ref{34}.
				\section{Conclusion}
	In this paper we extended the standard extremum seeking algorithms developed for online optimization to a class of online algorithms for optimization on Riemannian manifolds. We introduced the notion of geodesic dithers for extremum seeking algorithms on Riemannian manifolds and employed the results of averaging on manifolds to obtain a local convergence of the extremum seeking loop to a local optimizer on Riemannian manifolds. Two examples on Lie groups were presented to illustrate the efficacy of the proposed algorithm.  
		
	\appendix
	\section{Proof of Lemma \ref{kir} }
	\label{A0}
		The cost function $J$ may be expanded along geodesics by using the Taylor expansion on Riemannian manifolds, see \cite{smith}. We employ Lemma \ref{l1} to guarantee that there exist $a_{i}>0,\hspace{.2cm}i=1,\cdots,n$, such that\\ $\exp_{x}\left(\sum^{n}_{i=1}a_{i}\sin(\omega_{i}t)\frac{\partial}{\partial x_{i}}\right)\in \iota(x)$. Then the Taylor expansion of $J$ at $x\in M$ along the geodesic $\exp_{x}\left(\eta X\right)$, where $X\in T_{x}M$, is given by (see \cite{smith})
  \EQ \label{tay}&&\hspace{-.5cm}J(\exp_{x}\eta X)=J(x)+\eta (\nabla_{X}J)(x)+...+\frac{\eta^{m-1}}{(m-1)!}\times\nnum\\&&\hspace{0.5cm}(\nabla^{m-1}_{X}J)(x)+\frac{\eta^{m}}{(m-1)!}\int^{1}_{0}(1-s)^{m-1}\nabla^{m}_{X}J(\exp_{x}\left(s\eta X\right))ds, 0<\eta<\eta^{*},\EN
which is equivalent to  
 \EQ \label{tay2}&&\hspace{-.5cm}J(\exp_{x}\eta X)=J(x)+\eta (dJ(X))|_{x}+...+\frac{\eta^{m-1}}{(m-1)!}\times(\nabla^{m-2}_{X}dJ)(X)|_{x}+\nnum\\&&\hspace{-.5cm}\frac{\eta^{m}}{(m-1)!}\int^{1}_{0}(1-s)^{m-1}(\nabla^{m-1}_{X}dJ)(X)(\exp_{x}\left(s\eta X\right))ds,\quad 0<\eta<\eta^{*},\nnum\\\EN
 where $dJ:TM\rightarrow \mathds{R}$ is a differential form of $J$, $\eta^{*}$ is the upper existence limit for geodesics on $M$ and $\nabla$ is the \textit{Levi-Civita} connection, see \cite{Lee3}. Note that for compact manifolds $\eta^{*}=\infty$. The expansion above along the geodesic dithers in (\ref{gd}) is 
 \EQ &&J\left(\exp_{x}\left(\sum^{n}_{i=1}a_{i}\sin(\omega_{i}t)\frac{\partial}{\partial x_{i}}\right)\right)=J(x)+(\nabla_{\sum^{n}_{i=1}a_{i}\sin(\omega_{i}t)\frac{\partial}{\partial x_{i}}}J)(x)+\cdots+\nnum\\&&\frac{1}{(m-1)!}\times(\nabla^{m-1}_{\sum^{n}_{i=1}a_{i}\sin(\omega_{i}t)\frac{\partial}{\partial x_{i}}}J)(x)+\frac{1}{(m-1)!}\times\int^{1}_{0}(1-s)^{m-1}\nabla^{m}_{\sum^{n}_{i=1}a_{i}\sin(\omega_{i}t)\frac{\partial}{\partial x_{i}}}\nnum\\&&J(\exp_{x}s \sum^{n}_{i=1}a_{i}\sin(\omega_{i}t)\frac{\partial}{\partial x_{i}})ds.\nnum\EN
 Linear properties of $\nabla$ imply  that (see \cite{Lee3})
 \EQ \nabla_{\sum^{n}_{i=1}a_{i}\sin(\omega_{i}t)\frac{\partial}{\partial x_{i}}}J(x)=\sum^{n}_{i=1}a_{i}\sin(\omega_{i}t)\nabla_{\frac{\partial}{\partial x_{i}}}J(x),\nnum\EN
 and iteratively we have
 \EQ &&\nabla^{m}_{\sum^{n}_{i=1}a_{i}\sin(\omega_{i}t)\frac{\partial}{\partial x_{i}}}J(x)=\sum^{n}_{i=1}a_{i}\sin(\omega_{i}t)\nabla_{\frac{\partial}{\partial x_{i}}}\big(\nabla^{m-1}_{\sum^{n}_{j=1}a_{j}\sin(\omega_{j}t)\frac{\partial}{\partial x_{j}}}J\big)(x).\nnum\EN
We drop the notation $\hat{}$ for the state trajectory in (\ref{kk}) and the dynamical equations for the extremum seeking feedback loop are given in $x$ coordinates as follows:  
\EQ \label{koonkir}&&\dot{x}(t)=-\Big(\sum^{n}_{i=1}a_{i}\sin(\omega_{i}t)J(x)\frac{\partial}{\partial x_{i}}+\sum^{n}_{i,j=1}a_{i}a_{j}\sin(\omega_{i}t)\sin(\omega_{j}t)\nabla_{\frac{\partial}{\partial x_{j}}}J(x)\frac{\partial}{\partial x_{i}}+\cdots+\nnum\\&&\frac{1}{(m-1)!}\sum^{n}_{i,j=1}a_{i}a_{j}\sin(\omega_{i}t)\sin(\omega_{j}t)\times\nabla_{\frac{\partial}{\partial x_{j}}}\big(\nabla^{m-2}_{\sum^{n}_{l=1}a_{l}\sin(\omega_{l}t)\frac{\partial}{\partial x_{l}}}J\big)(x)\frac{\partial}{\partial x_{i}}+\nnum\\&&\frac{1}{(m-1)!}\sum^{n}_{i=1}a_{i}\sin(\omega_{i}t)\Big(\int^{1}_{0}(1-s)^{m-1}\nabla^{m}_{\sum^{n}_{j=1}a_{j}\sin(\omega_{j}t)\frac{\partial}{\partial x_{j}}}\nnum\\&&J\left(\exp_{x}\left(s \sum^{n}_{j=1}a_{j}\sin(\omega_{j}t)\frac{\partial}{\partial x_{j}}\right)\right)ds\Big)\frac{\partial}{\partial x_{i}}\Big).\nnum\\\EN 
 Denote the new time scale by $\tau\doteq\omega t$, then (\ref{koonkir}) is a $\tau$ varying vector field on $(M,g^{M})$ which is periodic with respect to $\tau$. The dynamical system (\ref{koonkir}) in $\tau$ scale is  given by
  \EQ &&\frac{dx}{d\tau}=-\frac{1}{\omega}\Big(\sum^{n}_{i=1}a_{i}\sin(\bar{\omega}_{i}\tau)J(x)\frac{\partial}{\partial x_{i}}+\sum^{n}_{i,j=1}a_{i}a_{j}\sin(\bar{\omega}_{i}\tau)\sin(\bar{\omega}_{j}\tau)\nabla_{\frac{\partial}{\partial x_{j}}}J(x)\frac{\partial}{\partial x_{i}}+\cdots+\nnum\\&&\frac{1}{(m-1)!}\sum^{n}_{i,j=1}a_{i}a_{j}\sin(\bar{\omega}_{i}\tau)\sin(\bar{\omega}_{j}\tau)\times \nabla_{\frac{\partial}{\partial x_{j}}}\big(\nabla^{m-2}_{\sum^{n}_{l=1}a_{l}\sin(\bar{\omega}_{l}\tau)\frac{\partial}{\partial x_{l}}}J\big)(x)\frac{\partial}{\partial x_{i}}+\nnum\\&&\frac{1}{(m-1)!}\sum^{n}_{i=1}a_{i}\sin(\bar{\omega}_{i}\tau)\Big(\int^{1}_{0}(1-s)^{m-1}\nabla^{m}_{\sum^{n}_{j=1}a_{j}\sin(\bar{\omega}_{j}\tau)\frac{\partial}{\partial x_{j}}}\nnum\\&&J\left(\exp_{x}\left(s \sum^{n}_{j=1}a_{j}\sin(\bar{\omega}_{j}\tau)\frac{\partial}{\partial x_{j}}\right)\right)ds\Big)\frac{\partial}{\partial x_{i}}\Big).\nnum\EN
 Let $T$ denote the least common multiplier of the periods of $\sin(\bar{\omega_{i}}\tau),\hspace{.2cm}i=1,\cdots,n$. 
 The average dynamical system is then given as
 \EQ \label{jj}&&\frac{dx}{d\tau}=\frac{1}{T}\int^{T}_{0}-\frac{1}{\omega}\Big(\sum^{n}_{i=1}a_{i}\sin(\bar{\omega}_{i}\tau)J(x)\frac{\partial}{\partial x_{i}}+\hspace{-.3cm}\sum^{n}_{i,j=1}a_{i}a_{j}\sin(\bar{\omega}_{i}\tau)\sin(\bar{\omega}_{j}\tau)\nabla_{\frac{\partial}{\partial x_{j}}}J(x)\frac{\partial}{\partial x_{i}}+\cdots+\nnum\\&&\frac{1}{(m-1)!}\sum^{n}_{i,j=1}a_{i}a_{j}\sin(\bar{\omega}_{i}\tau)\sin(\bar{\omega}_{j}\tau)\times \nabla_{\frac{\partial}{\partial x_{j}}}\big(\nabla^{m-2}_{\sum^{n}_{l=1}a_{l}\sin(\bar{\omega}_{l}\tau)\frac{\partial}{\partial x_{l}}}J\big)(x)\frac{\partial}{\partial x_{i}}+\nnum\\&&\frac{1}{(m-1)!}\sum^{n}_{i=1}a_{i}\sin(\bar{\omega}_{i}\tau)\Big(\int^{1}_{0}(1-s)^{m-1}\nabla^{m}_{\sum^{n}_{j=1}a_{j}\sin(\bar{\omega}_{j}\tau)\frac{\partial}{\partial x_{j}}}\nnum\\&&J\left(\exp_{x}\left(s \sum^{n}_{j=1}a_{j}\sin(\bar{\omega}_{j}\tau)\frac{\partial}{\partial x_{j}}\right)\right)ds\Big)\frac{\partial}{\partial x_{i}}\Big)d\tau,\nnum\\\EN
 Since $M$ is compact and $J$ is smooth then the higher derivatives of $J$ are all bounded above on $M$ and (\ref{jj}) is written as
 \EQ \label{ti}\frac{dx}{d\tau}=&&-\frac{1}{\omega}\sum^{n}_{i=1}\frac{a^{2}_{i}}{2}\nabla_{\frac{\partial}{\partial x_{i}}}J(x)\frac{\partial}{\partial x_{i}}+\frac{1}{\omega}\sum^{n}_{i=1}O((\max_{i\in1,\cdots,n}a_{i})^{4})\frac{\partial}{\partial x_{i}}.\EN
 The vector field $-\frac{1}{\omega}\sum^{n}_{i=1}\frac{a^{2}_{i}}{2}\nabla_{\frac{\partial}{\partial x_{i}}}J(x)\frac{\partial}{\partial x_{i}}+\frac{1}{\omega}\sum^{n}_{i=1}O((\max_{i\in1,\cdots,n}a_{i})^{4})\frac{\partial}{\partial x_{i}}$ is a perturbed version of the time invariant vector field $-\frac{1}{\omega}\sum^{n}_{i=1}\frac{a^{2}_{i}}{2}\nabla_{\frac{\partial}{\partial x_{j}}}J(x)\frac{\partial}{\partial x_{i}}$ on $(M,g^{M})$. Following (\ref{ggs}), we note that\\ $-\frac{1}{\omega}\sum^{n}_{i=1}\frac{a^{2}_{i}}{2}\nabla_{\frac{\partial}{\partial x_{j}}}J(x)\frac{\partial}{\partial x_{i}}$ is a scaled version of the gradient system presented in (\ref{ggs}). 

\section{Proof of Lemma \ref{kkir2}}
\label{A00}
Consider $J$  as the candidate Lyapunov function on $(M,g^{M})$. The variation of $J$ along $-\sum^{n}_{i=1}\frac{a^{2}_{i}}{2}\nabla_{\frac{\partial}{\partial x_{j}}}J(x)\frac{\partial}{\partial x_{i}}$ is 
 \EQ \mathcal{L}_{-\sum^{n}_{i=1}\frac{a^{2}_{i}}{2}\nabla_{\frac{\partial}{\partial x_{j}}}J(x)\frac{\partial}{\partial x_{i}}}J&&=dJ(-\sum^{n}_{i=1}\frac{a^{2}_{i}}{2}\nabla_{\frac{\partial}{\partial x_{j}}}J(x)\frac{\partial}{\partial x_{i}})\nnum\\&&=-\sum^{n}_{i=1}\frac{a^{2}_{i}}{2}dJ(\nabla_{\frac{\partial}{\partial x_{j}}}J(x)\frac{\partial}{\partial x_{i}}).\nnum\EN
 where $\mathcal{L}$ is the Lie derivative of $J$ along vector fields on $(M,g^{M})$. Locally $dJ=\sum^{n}_{i=1}\frac{\partial J}{\partial x_{i}}dx_{i}$, where $dx_{i}(\frac{\partial}{\partial x_{j}})=\delta_{i,j}$. Hence,
 \EQ \label{ly}\mathcal{L}_{-\sum^{n}_{i=1}\frac{a^{2}_{i}}{2}\nabla_{\frac{\partial}{\partial x_{j}}}J(x)\frac{\partial}{\partial x_{i}}}J=-\sum^{n}_{i=1}\frac{a^{2}_{i}}{2}(\frac{\partial J}{\partial x_{i}})^{2}\leq 0.\EN
Note that Assumption \ref{as} guarantees that, locally, \\$\sum^{n}_{i=1}\frac{a^{2}_{i}}{2}(\frac{\partial J}{\partial x_{i}})^{2}\ne 0$ for $x\ne x^{*}$, i.e. $\mathcal{
L}_{\acute{f}}J$ is locally negative-definite.

By Assumption \ref{as} and (\ref{ly}) the cost function $J:M\rightarrow \mathds{R}_{\geq 0}$ is locally positive definite, its derivative with respect to time is negative definite and $J(x^{*})=0$. Hence, $J$ is a Lyapunov function on $(M,g^{M})$, see \cite{Lewis}. Therefore, applying the results of \cite{Lewis}, Theorem 6.14, implies that $x^{*}$ is locally asymptotically stable.
\section{Proof of Theorem \ref{t1}}
\label{A000}

	We analyze closeness of solutions between state trajectories of  $\frac{dx}{d\tau}=\frac{1}{\omega} f(x,\tau)$ and state trajectories of $\frac{dx}{d\tau}=\frac{1}{\omega}\hat{f}(x)$, where \\$f(x,\tau)=-\sum^{n}_{i=1}\sin(\bar{\omega}_{i}\tau)J(\exp_{\hat{x}}\sum^{n}_{i=1}a\sin(\bar{\omega}_{i}\tau)\frac{\partial}{\partial x_{i}})\frac{\partial}{\partial x_{i}}$ and $\hat{f}(x)=\frac{1}{T}\int^{T}_{0}f(x,\tau)d\tau$. We consider the periodic vector field $Z$ defined in Lemma \ref{ll11}, $Z(t,x)\doteq\int^{t}_{0}(\hat{f}(x)-f(x,\tau))d\tau,\hspace{.2cm}x\in M,t\in\mathds{R}_{\geq 0}$,
 where $Z(t,x)=Z(t+T,x)$ and $T$ is the period of the extremum seeking system $f$. Now consider a composition of flows on $M$ given by
\EQ z(\tau)=\Phi^{(1,0)}_{\frac{1}{\omega} Z}\circ \Phi_{\frac{1}{\omega} f}(\tau,\tau_{0},x_{0}).\nnum\EN
By the results of Lemma \ref{ll11}, the tangent vector of $z$ is computed by
\EQ \label{rrr}\dot{z}(\tau)&=&T_{\Phi_{\frac{1}{\omega} f}(\tau,\tau_{0},x_{0})}\Phi_{\frac{1}{\omega} Z}^{(1,0)}\Big(\frac{1}{\omega} f(\Phi_{\frac{1}{\omega} f}(\tau,\tau_{0},x_{0}),\tau)\Big)+\frac{\partial}{\partial \tau}\big(\Phi_{\frac{1}{\omega} Z}^{(1,0)}\circ \Phi_{\frac{1}{\omega} f}(\tau,\tau_{0},x_{0})\big)\nnum\\&=&(\Phi^{-1})_{\frac{1}{\omega} Z}^{(1,0)^{*}}\Big(\frac{1}{\omega} f(\cdot,\tau)\Big)(z(\tau))+\frac{1}{\omega}\int^{1}_{0}(\Phi^{-1})^{(1,s)^{*}}_{\frac{1}{\omega} Z}\big(\hat{f}(\cdot)-f(\cdot,\tau)\big)ds\circ z(\tau),\nnum\\\EN 
where $(\Phi^{-1})^{(1,s)^{*}}_{\frac{1}{\omega} Z}$ is the pullback of the state flow $\Phi^{-1}_{\frac{1}{\omega} Z}$ and $\epsilon=\frac{1}{\omega}$. See Appendix \ref{A1} for the definition of pullbacks along diffeomorphisms.
Equivalently, in a compact form, we have
\EQ \label{kk1}\dot{z}(\tau)&=&\frac{1}{\omega}\Big[(\Phi^{-1})_{\frac{1}{\omega} Z}^{(1,0)^{*}} f+\int^{1}_{0}(\Phi^{-1})^{(1,s)^{*}}_{\frac{1}{\omega} Z}\big(\hat{f}-f\big)ds\Big]\circ z(\tau)\nnum\\&\doteq&\frac{1}{\omega} H\left(\frac{1}{\omega},\tau,z(\tau)\right). \EN
One may see that $H(0,\tau,x)=\hat{f}(x)$ where by the construction above, $H$ is smooth with respect to $\frac{1}{\omega}$. By applying the Taylor expansion with remainder we have
\EQ H\left(\frac{1}{\omega},\tau,x\right)=\hat{f}(x)+\frac{1}{\omega} h(x,\zeta,\tau),\nnum\EN 
where $h(x,\zeta,\tau)=\frac{\partial }{\partial \frac{1}{\omega}}H\left(\frac{1}{\omega},\tau,x\right)|_{\frac{1}{\omega}=\zeta}$ and $\zeta\in[0,\frac{1}{\omega}]$.
We note that $H\left(\frac{1}{\omega},\tau,x\right)$ is periodic with respect to $\tau$ since $f(x,\tau)$ and $Z(\tau,x)$ are both T-periodic. Hence, $h(x,\zeta,\tau)$ is a T-periodic vector field on $M$. 

The metric triangle inequality on $(M,g^{M})$ implies
\EQ \label{kirkir}&&d(\Phi_{\frac{1}{\omega} f}(\tau,\tau_{0},x_{0}), \Phi_{\frac{1}{\omega} \hat{f}}(\tau,\tau_{0},x_{0}))\leq d(\Phi_{\frac{1}{\omega} f}(\tau,\tau_{0},x_{0}),\Phi_{\frac{1}{\omega} Z}^{(1,0)}\circ \Phi_{\frac{1}{\omega} f}(\tau,\tau_{0},x_{0}))+\nnum\\&& d(\Phi_{\frac{1}{\omega} Z}^{(1,0)}\circ \Phi_{\frac{1}{\omega} f}(\tau,\tau_{0},x_{0}),\Phi_{\frac{1}{\omega} \hat{f}}(\tau,\tau_{0},x_{0}))\leq
d(\Phi_{\frac{1}{\omega} f}(\tau,\tau_{0},x_{0}),\Phi_{\frac{1}{\omega} Z}^{(1,0)}\circ \Phi_{\frac{1}{\omega} f}(\tau,\tau_{0},x_{0}))+\nnum\\&& d(\Phi_{\frac{1}{\omega} Z}^{(1,0)}\circ \Phi_{\frac{1}{\omega} f}(\tau,\tau_{0},x_{0}),x^{*})+d(\Phi_{\frac{1}{\omega} \hat{f}}(\tau,\tau_{0},x_{0}),x^{*}).\nnum\\\EN

Based on (\ref{kirkir}), We analyze the closeness of solutions for the following dynamics.
\EQ \label{kkk}&&\dot{x}(\tau)=\frac{1}{\omega} \hat{f}(x(\tau)),\hspace{2.5cm}x(\tau_{0})=x_{0},\nnum\\&&\dot{y}(\tau)=\frac{1}{\omega} \acute{f}(y(\tau)),\hspace{2.5cm}y(\tau_{0})=x_{0},\nnum\\&&\dot{z}(\tau)=\frac{1}{\omega} \hat{f}(z(\tau))+\frac{1}{\omega^{2}} h(z,\zeta,\tau),\hspace{.2cm}z(\tau_{0})=x_{0},\EN
where $\acute{f}=-\sum^{n}_{i=1}\frac{a^{2}_{i}}{2}\nabla_{\frac{\partial}{\partial x_{j}}}J(x)\frac{\partial}{\partial x_{i}}$ is the gradient system.
Rescaling time  back to $t$ via $t=\frac{1}{\omega}\tau$, yields
\EQ &&\frac{dx}{dt}=\hat{f}\left(x(t)\right),\hspace{2.2cm}x(t_{0})=x_{0},\nnum\\&&\frac{dy}{dt}=\acute{f}\left(y(t)\right),\hspace{2.2cm}y(t_{0})=x_{0},\nnum\\&&\frac{dz}{dt}= \hat{f}\left(z(t)\right)+\frac{1}{\omega} h(z,\zeta,t),\hspace{.2cm}z(t_{0})=x_{0},\nnum\EN
or equivalently by Lemma \ref{kir}
\EQ \label{kirekhar}&&\frac{dx}{dt}=\acute{f}\left(x(t)\right)+\sum^{n}_{i=1}O((\max_{i\in1,\cdots,n}a_{i})^{4})\frac{\partial}{\partial x_{i}},\nnum\\&&\frac{dy}{dt}=\acute{f}\left(y(t)\right),\nnum\\&&\frac{dz}{dt}= \acute{f}\left(z(t)\right)+\sum^{n}_{i=1}O((\max_{i\in1,\cdots,n}a_{i})^{4})\frac{\partial}{\partial z_{i}}+\frac{1}{\omega} h(z,\zeta,t),\EN
where $x(t_{0})=y(t_{0})=z(t_{0})=x_{0}$.

The variation of the cost function $J$ along $\hat{f}(\cdot)$ is given by
\EQ \mathcal{L}_{\hat{f}}J&&=\mathcal{L}_{\acute{f}+\sum^{n}_{i=1}O((\max_{i\in1,\cdots,n}a_{i})^{4})\frac{\partial}{\partial x_{i}}}J\nnum\\&&=\mathcal{L}_{\acute{f}}J+\mathcal{L}_{\sum^{n}_{i=1}O((\max_{i\in1,\cdots,n}a_{i})^{4})\frac{\partial}{\partial x_{i}}}J\nnum\\&&=\mathcal{L}_{\acute{f}}J+\sum^{n}_{i=1}O((\max_{i\in1,\cdots,n}a_{i})^{4})\mathcal{L}_{\frac{\partial}{\partial x_{i}}}J.\nnum\EN 

As shown by the proof of Lemma \ref{kkir2}, locally, we have $\mathcal{L}_{\acute{f}}J\leq 0$.
Without loss of generality, assume positive definiteness and negative definiteness of $J$ and  \\$\mathcal{L}_{\acute{f}}J=\mathcal{L}_{-\frac{1}{\omega}\sum^{n}_{i=1}\frac{a^{2}_{i}}{2}\nabla_{\frac{\partial}{\partial x_{j}}}J(x)\frac{\partial}{\partial x_{i}}}J$ are both obtained on  $U_{x^{*}}\subset M$ of $x^{*}$. Otherwise we apply the intersection of the corresponding neighborhoods to perform the analysis above.
Define the sublevel set $\mathcal{N}_{b}$ of the positive definite function $J:M\rightarrow\mathds{R}_{\geq 0}$  as $\mathcal{N}_{b}\doteq\{x\in M, \hspace{.2cm} J(x)\leq b\}$. By $\mathcal{N}_{b}(x^{*})$ we denote a connected sublevel set of $M$ containing $x^{*}\in M$.
By Lemma 6.12 in \cite{Lewis}, there exists a subslevel set $\mathcal{N}_{b}(x^{*})\subset U_{x^{*}}$, such that $\mathcal{N}_{b}(x^{*})$ is compact.
Consider a neighborhood of $x^{*}$ denoted by $W_{x^{*}}$ such that $W_{x^{*}}\subset int(\mathcal{N}_{b}(x^{*}))\subset U_{x^{*}}\subset M$, where $int()$ gives the interior set. Compactness of $M$ implies that  $M-W_{x^{*}}$ is closed and compact. Hence, $\mathcal{L}_{\acute{f}}J<0$ for all $x\in (M-W_{x^{*}})\bigcap \mathcal{N}_{b}(x^{*})$. Note that $\mathcal{L}_{\acute{f}}J<0$ on $U_{x^{*}}-\{x^{*}\}$.

  Smoothness of $J$ and compactness of $M-W_{x^{*}}$ together imply that $\mathcal{L}_{\acute{f}}J$ attains its bounded maximum value in $M-W_{x^{*}}$. Hence, by selecting $a_{i},\hspace{.2cm}i=1,\cdots,n$ sufficiently small we have $\mathcal{L}_{\hat{f}}J<0$ on $(M-W_{x^{*}})\bigcap\mathcal{N}_{b}(x^{*})$. This implies that the state trajectory $\Phi_{\acute{f}}(t,t_{0},x_{0})$ remains in $\mathcal{N}_{b}(x^{*})$ for $x_{0}\in int(\mathcal{N}_{b}(x^{*}))$.

The variation of $J$ along $ \hat{f}\left(z(t)\right)+\frac{1}{\omega} h(z,\zeta,t)$ is then given by 

\EQ \mathcal{L}_{\hat{f}+\frac{1}{\omega} h}J&=&\mathcal{L}_{\hat{f}}J+\frac{1}{\omega} \mathcal{L}_{h}J\nnum\\&=&\mathcal{L}_{\acute{f}}J+\sum^{n}_{i=1}O((\max_{i\in1,\cdots,n}a_{i})^{4})\mathcal{L}_{\frac{\partial}{\partial x_{i}}}J+\frac{1}{\omega} \mathcal{L}_{h}J.\nnum\EN
The same argument applies  to the variation of $J$ along $\hat{f}+\frac{1}{\omega}h(z,\zeta,t)$ and we obtain that $\mathcal{L}_{\hat{f}+\frac{1}{\omega} h}J<0$ on $(M-W_{x^{*}})\bigcap \mathcal{N}_{b}(x^{*})$ for sufficiently small $a_{i}$ and sufficiently large $\omega$. Note that $h$ is periodic with respect to $t$, $\zeta\in[0,\frac{1}{\omega}]$ and $M$ is compact. Hence, $\mathcal{L}_{h}J$ is bounded and  this implies that by choosing $a_{i}$ sufficiently small and $\omega$ sufficiently large the state trajectory $z(\cdot)$ remains in $\mathcal{N}_{b}(x^{*})$ for all initial states $z_{0}\in int(\mathcal{N}_{b}(x^{*}))$.

Denote the uniform normal neighborhood of $x^{*}\in M$ with respect to $U_{x^{*}}$  by $U^{n}_{x^{*}}$ (its existence is guaranteed by Lemma 5.12 in \cite{Lee3}). Consider a  geodesic ball of radius $\delta$ where $U^{n}_{x^{*}}\subset \exp_{x^{*}}(B_{\delta}(0)) \subset U_{x^{*}}$.   By definition, $\exp_{x^{*}}(B_{\delta}(0))$ is an open set containing $x^{*}$ in the topology of $M$. Therefore  one may shrink $b$ to $\acute{b}, 0<\acute{b}\leq b,$  such that  $\mathcal{N}_{\acute{b}}(x^{*})\subset \exp_{x^{*}}(B_{\delta}(0))$. Hence, by the argument above, we select the set of initial state such as $\Phi_{\hat{f}+\frac{1}{\omega}h}(\cdot,t_{0},x_{0})$ stays in a normal neighborhood of $x^{*}$. For the economy of notation we replace $\acute{b}$ with $b$ and assume $\mathcal{N}_{b}(x^{*})\subset \exp_{x^{*}}(B_{\delta}(0))\subset U_{x^{*}}$.

We analyze the distance between the state trajectory $y(\cdot)$ of the asymptotically stable system and the averaged and full systems. As is obvious from (\ref{kirekhar}), the averaged systems and the dynamical system corresponding to $z(\cdot)$ are  perturbations of  the gradient system $\dot{y}=\acute{f}(y)$, where $\dot{y}=\acute{f}(y)$ is locally asymptotically stable and the magnitude of the  perturbations vector fields  is arbitrarily shrunken by adjusting $a_{i}$ and $\omega$ in (\ref{kirekhar}). Since the initial state set is chosen such that the state trajectory $z(\cdot)$ remains in a normal neighborhood of $x^{*}$, then the conditions of Theorems \ref{t7} in Appendix \ref{A2} are satisfied. Hence, there exist  a neighborhood $U^{1}_{x^{*}}\subset int(\mathcal{N}_{b}(x^{*}))$ and a continuous function $\rho$, such that for all $x_{0}\in U^{1}_{x^{*}}$
 \EQ \label{fuck1}&&\limsup_{t\rightarrow \infty} d\left(\Phi_{\hat{f}+\frac{1}{\omega}h}(t,t_{0},x_{0}),x^{*}\right)\leq\nnum\\&&\rho\left(\sup_{z\in M, t\in [t_{0},t_{0}+T]}||\sum^{n}_{i=1}O((\max_{i\in 1,\cdots,n}a_{i})^{4})\frac{\partial}{\partial z{i}}+\frac{1}{\omega}h(z,\zeta,t)||_{g^{M}}\right),\nnum\\\EN
 where $\rho$ is a continuous function which is zero at zero. We note that since $M$ is compact, $h$ is periodic with respect to $t$ and $\zeta\in[0,\frac{1}{\omega}]$, then \\ $\sup_{z\in M, t\in [t_{0},t_{0}+T]}||\sum^{n}_{i=1}O((\max_{i\in 1,\cdots,n}a_{i})^{4})\frac{\partial}{\partial z{i}}+\frac{1}{\omega}h(z,\zeta,t)||_{g^{M}}$ is bounded.

Also note that (\ref{fuck1}) does not guarantee the convergence of the perturbed state trajectory to $x^{*}$. However, it gives a local closeness of solutions in terms of the Riemannian distance function $d$ to $x^{*}$ after elapsing enough time. The closeness estimation provided in  (\ref{fuck1}) is used to bound the distance between $y(\cdot)$, $x(\cdot)$ and $z(\cdot)$ as follows.
By employing the triangle inequality we have 
\EQ \label{kirkir1}&&d\left(\Phi_{ f}(t,t_{0},x_{0}),\Phi_{\acute{f}}(t,t_{0},x_{0})\right)\leq d\left(\Phi_{ f}(t,t_{0},x_{0}),\Phi_{\hat{f}+\frac{1}{\omega}h}(t,t_{0},x_{0})\right)+\nnum\\&&d\left(\Phi_{\hat{f}+\frac{1}{\omega}h}(t,t_{0},x_{0}),\Phi_{\acute{f}}(t,t_{0},x_{0})\right),\EN
and
\EQ \label{kirkir2}&&d\left(\Phi_{\hat{f}+\frac{1}{\omega}h}(t,t_{0},x_{0}),\Phi_{\acute{f}}(t,t_{0},x_{0})\right)\leq d\left(\Phi_{ \hat{f}+\frac{1}{\omega}h}(t,t_{0},x_{0}),x^{*}\right)+d\left(x^{*},\Phi_{\acute{f}}(t,t_{0},x_{0})\right),\nnum\\\EN
 where in  (\ref{kirkir2}), $d(x^{*},\Phi_{\acute{f}}(t,t_{0},x_{0}))$ converges to zero and  $d(\Phi_{ \hat{f}+\frac{1}{\omega}h}(t,t_{0},x_{0}),x^{*})$ can be chosen arbitrarily small by (\ref{fuck1}). Note that $||h(z,\zeta,t)||_{g^{M}}$ is bounded since $M$ is compact, $\zeta\in[0,\frac{1}{\omega}]$ and $h$ is periodic with respect to $t$.
	In order to show the boundedness of $d(\Phi_{ f}(t,t_{0},x_{0}),\Phi_{\hat{f}+\frac{1}{\omega}h}(t,t_{0},x_{0}))$ in (\ref{kirkir1}), we switch back to the time scale $\tau$.
Now we prove $d(\Phi_{\frac{1}{\omega} f}(\tau,\tau_{0},x_{0}),\Phi_{\frac{1}{\omega} Z}^{(1,0)}\circ \Phi_{\frac{1}{\omega} f}(\tau,\tau_{0},x_{0}))=O\left(\frac{1}{\omega}\right)$.

By the definition of the distance function given in (\ref{dist}), we have  $d(\Phi_{\frac{1}{\omega}Z}(s,0,x),x)\leq \ell(\Phi_{\frac{1}{\omega} Z}(s,0,x)),$ where $\ell(\Phi_{\frac{1}{\omega} Z}(s,0,x))$ is the length of the curve connecting $x$ to $\Phi_{\frac{1}{\omega} Z}(s,0,x)$ on $M$. Therefore,
\EQ \label{kirkoon}&&d\left(\Phi_{\frac{1}{\omega} Z}(1,0,x),x\right)\leq\ell(\Phi_{\frac{1}{\omega} Z}(1,0,x))= \frac{1}{\omega}\int^{1}_{0}||Z(t,\Phi_{\frac{1}{\omega}Z}(s,0,x))||_{g^{M}}ds.\EN
Periodicity of $Z$ with respect to $t$, boundedness of $\Phi_{\frac{1}{\omega}Z}(\cdot,0,x)$ in the sense of compactness of $M$ and smoothness of $Z$ with respect to $x$ together yield $d(\Phi_{\frac{1}{\omega} Z}(1,0,x),x)=O\left(\frac{1}{\omega}\right)$. Since $x$ is a generic element of $M$ we have 
\EQ d\left(\Phi_{\frac{1}{\omega} f}(\tau,\tau_{0},x_{0}),\Phi_{\frac{1}{\omega} Z}^{(1,0)}\circ \Phi_{\frac{1}{\omega} f}(\tau,\tau_{0},x_{0})\right)=O\left(\frac{1}{\omega}\right),\forall \tau\in[\tau_{0},\infty),\nnum\EN
where $x$ is replaced by $\Phi_{\frac{1}{\omega} f}(\tau,\tau_{0},x_{0})\in M$.
Hence, by using (\ref{kirkir}), for any $x_{0}\in U^{1}_{x^{*}}$, there exists a time $\acute{T}_{x_{0}}$, such that
\EQ &&d\left(\Phi_{\frac{1}{\omega} f}(\tau,\tau_{0},x_{0}), \Phi_{\frac{1}{\omega} \acute{f}}(\tau,\tau_{0},x_{0})\right)\leq d\left(\Phi_{\frac{1}{\omega} f}(\tau,\tau_{0},x_{0}),\Phi_{\frac{1}{\omega} Z}^{(1,0)}\circ \Phi_{\frac{1}{\omega} f}(\tau,\tau_{0},x_{0})\right)+\nnum\\&& d\left(\Phi_{\frac{1}{\omega} Z}^{(1,0)}\circ \Phi_{\frac{1}{\omega} f}(\tau,\tau_{0},x_{0}),\Phi_{\frac{1}{\omega} \acute{f}}(\tau,\tau_{0},x_{0})\right)\leq d\left(\Phi_{\frac{1}{\omega} f}(\tau,\tau_{0},x_{0}),\Phi_{\frac{1}{\omega} Z}^{(1,0)}\circ \Phi_{\frac{1}{\omega} f}(\tau,\tau_{0},x_{0})\right)+\nnum\\&&d\left(\Phi_{\frac{1}{\omega} Z}^{(1,0)}\circ \Phi_{\frac{1}{\omega} f}(\tau,\tau_{0},x_{0}),x^{*}\right)+d\left(x^{*},\Phi_{\frac{1}{\omega} \acute{f}}(\tau,\tau_{0},x_{0})\right)\leq \nnum\\&& O\left(\frac{1}{\omega}\right)+\rho\left(\sup_{z\in M,t\in[t_{0},t_{0}+T]}||\sum^{n}_{i=1}O((\max_{i\in 1,\cdots,n}a_{i})^{4})\frac{\partial}{\partial z{i}}+\frac{1}{\omega}h(z,\zeta,t)||_{g^{M}}\right)+\nnum\\&&d\left(x^{*},\Phi_{\frac{1}{\omega} \acute{f}}(\tau,\tau_{0},x_{0})\right),\hspace{.2cm}\forall \tau\in [\omega \acute{T}_{x_{0}},\infty).\nnum\EN
Note that $\Phi_{\frac{1}{\omega} Z}^{(1,0)}\circ \Phi_{\frac{1}{\omega} f}(\tau,\tau_{0},x_{0})=\Phi_{f+\frac{1}{\omega}h}(t,t_{0},x_{0})$, for $\tau=\omega t$ and $\tau_{0}=\omega t_{0}$.
Finally we have 
\EQ \label{koonak}&&d\left(\Phi_{\frac{1}{\omega} f}(\tau,\tau_{0},x_{0}),x^{*}\right)\leq d\left(\Phi_{\frac{1}{\omega} f}(\tau,\tau_{0},x_{0}), \Phi_{\frac{1}{\omega} \acute{f}}(\tau,\tau_{0},x_{0})\right)+ d\left(x^{*}, \Phi_{\frac{1}{\omega} \acute{f}}(\tau,\tau_{0},x_{0})\right)\leq \nnum\\&&O\left(\frac{1}{\omega}\right)+\rho\left(\sup_{z\in M,t\in[t_{0}.t_{0}+T]}||\sum^{n}_{i=1}O((\max_{i\in 1,\cdots,n}a_{i})^{4})\frac{\partial}{\partial z{i}}+\frac{1}{\omega}h(z,\zeta,t)||_{g^{M}}\right)+\nnum\\&&2d\left(x^{*},\Phi_{\frac{1}{\omega} \acute{f}}(\tau,\tau_{0},x_{0})\right),\hspace{.2cm}\forall \tau\in [\omega\acute{T}_{x_{0}},\infty), x_{0}\in U^{1}_{x^{*}}.
\nnum\\\EN
 As (\ref{koonak}) indicates  $d\left(\Phi_{\frac{1}{\omega} f}(\tau,\tau_{0},x_{0}),x^{*}\right)$ can be ultimately bounded by shrinking $a_{i},\hspace{.2cm}i=1,\cdots,n$ and increasing $\omega$ such that the state trajectory $\Phi_{f}(t,t_{0},x_{0})$  enters $U_{x^{*}}$ and remains there.  

	\section{Averaging on Riemannian manifolds}
	\label{A1}
	
Let us consider a perturbed system as
\EQ \label{pp}\dot{x}(t)=\epsilon f(x(t),t), \hspace{.2cm}f\in \mathfrak{X}(M\times \mathds{R}),x_{0}\in M,\hspace{.2cm}\epsilon\geq 0,\nnum\EN
where $f$ is periodic in $t$ with the period $T$, i.e. $f(x,t)=f(x,t+T)$. Such a system is referred to as \textit{$T$-periodic}. The averaged vector field $\hat{f}$ is given by 
\EQ \label{ppp}\hat{f}(x)\doteq\frac{1}{T}\int^{T}_{0}f(x,s)ds,\EN
where the average dynamical system is locally given by $\dot{x}(t)=\epsilon \hat{f}(x(t))$.
	
In order to obtain  closeness of solutions for dynamical systems we employ the notion of pullbacks of vector fields along diffeomorphisms on 
$M$ as per Definition \ref{koskalak}. 

We have the following lemma for the variation of smooth parameter varying vector fields.
\begin{lemma}[\hspace{-.02cm}\cite{Agra}, Page 40, \cite{Lewis}, Page 451]
\label{ll111}
Consider a smooth parameter varying vector field $Y(\lambda,x)$, where $Y\in \mathfrak{X}(\mathds{R}\times M)$ with the associated flow $\Phi_{Y}(t,t_{0},\cdot):M\rightarrow M$. Then,
\EQ\label{poo}&&\frac{\partial}{\partial\lambda}\Phi_{Y}(t,t_{0},x_{0})=T_{x_{0}}\Phi^{(t,t_{0})}_{Y}\int^{t}_{t_{0}}\left(\Phi^{(s,t_{0})^{*}}_{Y}\frac{\partial }{\partial \lambda}Y\left(\lambda,\cdot\right)\right)(x_{0})ds=\nnum\\&&\int^{t}_{t_{0}}(\Phi^{-1})^{(t,s)^{*}}_{Y}\frac{\partial }{\partial \lambda}Y(\lambda,\cdot)ds\circ \Phi_{Y}(t,t_{0},x_{0}) \in T_{\Phi_{Y}(t,t_{0},x_{0})}M,\EN
where $T_{x_{0}}\Phi^{(t,t_{0})}_{Y}$ is the pushforward of $\Phi_{Y}(t,t_{0},\cdot)$ at $x_{0}$.
\end{lemma}

	\subsection{Proof of Lemma \ref{ll11}}
	The proof parallels the results of \cite{Lewis}, Chapter 9 on $\mathds{R}^{n}$. We compute the tangent vector of $z(t)\doteq\Phi_{\epsilon Z}^{(1,0)}\circ \Phi_{\epsilon f}(t,t_{0},x_{0})\in M$ where $Z(t,x)\doteq\int^{t}_{0}\big(\hat{f}(x)-f(x,s)\big)ds,\hspace{.2cm}0\leq t,$. The derivative of $\Phi_{\epsilon Z}^{(1,0)}\circ \Phi_{\epsilon f}(t,t_{0},x_{0})$ with respect to time has two components as follows:
\EQ\label{rr}&&\dot{z}(t)=T_{\Phi_{\epsilon f}(t,t_{0},x_{0})}\Phi_{\epsilon Z}^{(1,0)}\Big(\epsilon f(\Phi_{\epsilon f}(t,t_{0},x_{0}),t)\Big)+\frac{\partial}{\partial t}\big(\Phi_{\epsilon Z}^{(1,0)}\circ \Phi_{\epsilon f}(t,t_{0},x_{0})\big), \EN
where $\epsilon f(\Phi_{\epsilon f}(t,t_{0},x_{0}),t)=\frac{\partial}{\partial t}\Phi_{\epsilon f}(t,t_{0},x_{0})$.
The first term is the variation of $z(t)$ with respect to the variation of the initial state $\Phi_{\epsilon f}(t,t_{0},x_{0})$ and the second term is the variation of $z(t)$ with respect to the variation of $Z$. 
Note that the flow $\Phi_{\epsilon Z}^{(s,0)}\circ \Phi_{\epsilon f}(t,t_{0},x_{0})$ has two time scales $t$ and $s$ which are independent. Hence, the vector field $Z$ is $s$ invariant and $t$ dependent where $t$ appears as a parameter in $Z$.

By (\ref{pu}) we have
\EQ  T_{\Phi_{\epsilon f}(t,t_{0},x_{0})}\Phi_{\epsilon Z}^{(1,0)}\Big(\epsilon f(\Phi_{\epsilon f}(t,t_{0},x_{0}),t)\Big)&=&(\Phi^{-1})_{\epsilon Z}^{(1,0)^{*}}\Big(\epsilon f(\cdot,t)\Big)(\Phi_{\epsilon Z}^{(1,0)}\circ \Phi_{\epsilon f}(t,t_{0},x_{0}))\nnum\\&=&(\Phi^{-1})_{\epsilon Z}^{(1,0)^{*}}\Big(\epsilon f(\cdot,t)\Big)(z(t)),\nnum \EN
and by Lemma \ref{ll111}, we have
\EQ &&\frac{\partial}{\partial t}\big(\Phi_{\epsilon Z}^{(1,0)}\circ
\Phi_{\epsilon
f}(t,t_{0},x_{0})\big)=\int^{1}_{0}(\Phi^{-1})^{(1,s)^{*}}_{\epsilon
Z}\frac{\partial}{\partial t}\big(Z(t,\cdot)\big)ds\circ\Phi_{\epsilon Z}^{(1,0)}\circ \Phi_{\epsilon f}(t,t_{0},x_{0}))=\nnum\\&&\epsilon\int^{1}_{0}(\Phi^{-1})^{(1,s)^{*}}_{\epsilon
Z}\big(\hat{f}(\cdot)-f(\cdot,t)\big)ds\circ z(t),\nnum\EN

 Hence, 
\EQ \label{rrrr}&&\dot{z}(t)=(\Phi^{-1})_{\epsilon Z}^{(1,0)^{*}}\Big(\epsilon f(\cdot,t)\Big)(z(t))+\epsilon\int^{1}_{0}(\Phi^{-1})^{(1,s)^{*}}_{\epsilon Z}\big(\hat{f}(\cdot)-f(\cdot,t)\big)ds\circ z(t).\EN 
In a compact form, (\ref{rrrr}) is written as
\EQ \dot{z}(t)&=&\epsilon\Big[(\Phi^{-1})_{\epsilon Z}^{(1,0)^{*}} f+\int^{1}_{0}(\Phi^{-1})^{(1,s)^{*}}_{\epsilon Z}\big(\hat{f}-f\big)ds\Big]\circ z(t)\nnum\\&\doteq&\epsilon H(\epsilon,t,z(t))\in T_{z(t)}M.\nnum\EN

	\section{Proof of Theorem \ref{tt11}}
\label{AK}
The proof parallels the proof of Theorem \ref{t1} by employing the results of Lemmas \ref{kakir2} and \ref{kirak}. However, (\ref{kirkoon}) does not necessarily hold since $G$ is not compact. The same as (\ref{kirkoon}) we observe that
	\EQ &&d(\Phi_{ f}(t,t_{0},g_{0}),\Phi_{\acute{f}}(t,t_{0},g_{0}))\leq d(\Phi_{ f}(t,t_{0},g_{0}),\Phi_{\hat{f}+\frac{1}{\omega}h}(t,t_{0},g_{0}))+\nnum\\&&d(\Phi_{\hat{f}+\frac{1}{\omega}h}(t,t_{0},g_{0}),\Phi_{\acute{f}}(t,t_{0},g_{0})),\nnum\EN
and
\EQ &&d(\Phi_{\hat{f}+\frac{1}{\omega}h}(t,t_{0},g_{0}),\Phi_{\acute{f}}(t,t_{0},g_{0}))\leq d(\Phi_{ \hat{f}+\frac{1}{\omega}h}(t,t_{0},g_{0}),g^{*})+d(g^{*},\Phi_{\acute{f}}(t,t_{0},g_{0})),\nnum\EN
	where $f$, $\hat{f}$ and $\acute{f}$ are the extremum seeking, averaged and gradient vector fields induced by (\ref{kkg}). By the results of Lemma \ref{kakir2} and Theorem \ref{t7},  $d(g^{*},\Phi_{\acute{f}}(t,t_{0},g_{0}))$ converges to zero and $d(\Phi_{ \hat{f}+\frac{1}{\omega}h}(t,t_{0},g_{0}),g^{*})$ is ultimately bounded by a continuous function $\rho$ where the bound is shrunken by adjusting $a_{i}, i=1,\cdots,n$ and $\omega$. Note that in the proof of Theorem \ref{t7} it is shown that asymptotic stability of the gradient vector field $\acute{f}$ guarantees that the state trajectories of $\hat{f}$ and $\hat{f}+\frac{1}{\omega}h$ remains in a compact subset containing the equilibrium of $\acute{f}$.
		
		It remains to show $d(\Phi_{\frac{1}{\omega} f}(\tau,\tau_{0},x_{0}),\Phi_{\frac{1}{\omega} Z}^{(1,0)}\circ \Phi_{\frac{1}{\omega} f}(\tau,\tau_{0},x_{0}))=O\left(\frac{1}{\omega}\right)$ in the time scale $\tau=\omega t$ for $\tau\in[\tau_{0},\infty)$.
		
		Following the proof of Theorem \ref{t7} one may show $\Phi_{\hat{f}+\epsilon h}(\omega,\omega_{0},x_{0})\in \mathcal{N}_{b}(g^{*}),\hspace{.2cm}\omega\in[\omega_{0},\infty), g_{0}\in int(\mathcal{N}_{b}(g^{*}))$ for some $b\in \mathds{R}>0$, where $\mathcal{N}_{b}(g^{*})$ is a compact connected sublevel set of a Lyapunov function containing $g^{*}\in G$, see the proof of Theorem \ref{t7} and \cite{Lewis} Lemma 6.12.

		Hence,  $\Phi_{\frac{1}{\omega} Z}^{(1,0)}\circ \Phi_{\frac{1}{\omega} f}(t,t_{0},x_{0})\in \mathcal{N}_{b}(x^{*}),\hspace{.2cm}\tau\in[\tau_{0},\infty),g_{0}\in int(\mathcal{N}_{b}(g^{*}))$.
Therefore,
\EQ  \hspace{-.5cm}\bigcup_{\tau\in[\tau_{0},\infty)}\Phi_{\frac{1}{\omega} f}(\tau,\tau_{0},g_{0})&\subset& \bigcup_{\tau\in[\tau_{0},\infty)}(\Phi^{-1})_{\frac{1}{\omega} Z}^{(1,0)}\circ\mathcal{N}_{b}(g^{*})\nnum\\&=&\bigcup_{\tau\in[\tau_{0},\tau_{0}+\omega T]}(\Phi^{-1})_{\frac{1}{\omega} Z}^{(1,0)}\circ\mathcal{N}_{b}(g^{*})\nnum\\&\subset & \bigcup_{\tau\in[\tau_{0},\tau_{0}+\omega T],\frac{1}{\omega}\in[0,\hat{\epsilon}]}(\Phi^{-1})_{\frac{1}{\omega} Z}^{(1,0)}\circ\mathcal{N}_{b}(x^{*}) ,\nnum \EN
where the equality is due the periodicity of $Z$ and $\hat{\epsilon}$ is a limit for the minimum frequency ($0<\omega<\infty$). 
 Compactness of $\mathcal{N}_{b}(g^{*})$, $[\tau_{0},\tau_{0}+\omega T]$ and $[0,\hat{\epsilon}]$ together with the smoothness of $\Phi^{-1}$ gives compactness of $\bigcup_{\tau\in[\tau_{0},\tau_{0}+\omega T],\frac{1}{\omega}\in[0,\hat{\epsilon}]}(\Phi^{-1})_{\frac{1}{\omega} Z}^{(1,0)}\circ\mathcal{N}_{b}(g^{*})$ in $G$.

Hence, we show that 
\EQ \label{koonii}&& d(\Phi_{\frac{1}{\omega} Z}(1,0,z),z)=O\left(\frac{1}{\omega}\right), \hspace{.2cm}\nnum\\&&z\in \bigcup_{\tau\in[\tau_{0},\tau_{0}+\omega T],\frac{1}{\omega}\in[0,\hat{\epsilon}]}(\Phi^{-1})_{\frac{1}{\omega} Z}^{(1,0)}\circ\mathcal{N}_{b}(g^{*}),\EN
which proves
\EQ d(\Phi_{\frac{1}{\omega} f}(\tau,\tau_{0},x_{0}),\Phi_{\frac{1}{\omega} Z}^{(1,0)}\circ \Phi_{\frac{1}{\omega} f}(\tau,\tau_{0},x_{0}))=O\left(\frac{1}{\omega}\right),\nnum\EN
where $z$ is replaced by $\Phi_{\frac{1}{\omega} f}(\tau,\tau_{0},x_{0})$ in (\ref{koonii}). 
The rest of the proof is identical to the proof of Theorem \ref{t1}.		
	\section{Stability of perturbed systems on Riemannian manifolds}
	\label{A2}
Consider the following perturbed dynamical system on $(M,g^{M})$.
\EQ \label{per}\dot{x}(t)=f(x,t)+h(x,t),f,h\in \mathfrak{X}(M\times \mathds{R}).\EN
The term $h$ is considered as a perturbation of the nominal system $f$. The next lemma gives the existence of  Lyapunov functions for dynamical systems on Riemannian manifold which satisfy specific local properties. 
	\begin{lemma}[\hspace{-.01cm}\cite{Taringoo100}]
\label{lf1}
Let $x^{*}$ be an  equilibrium for the smooth dynamical system $\dot{x}=f(x,t)$ which is uniformly asymptotically stable (see \cite{Kha}) on an open set $\mathcal{N}_{x^{*}}\subset \mathcal{U}^{n}_{x^{*}}$ ($\mathcal{U}^{n}_{x^{*}}$ is a normal neighborhood around $x^{*}$).
 Assume $||T_{x}f(\cdot,t)||$ is uniformly bounded with respect to $t$ on $\mathcal{N}_{x^{*}}$, where $||.||$ is the norm of the bounded linear operator $Tf:TM\rightarrow TTM$. Then, for some $\mathcal{U}_{x^{*}}\subset \mathcal{U}^{n}_{x^{*}}$, for all $x(t_{0})=x_{0}\in \mathcal{U}_{x^{*}}$, there exist a differentiable function $w:M\times \mathds{R}\rightarrow \mathds{R}_{\geq 0}$ and $\alpha_{1},\alpha_{2},\alpha_{3},\alpha_{4}\in \mathcal{K}$ (continuous, strictly increasing and zero at zero, see \cite{Kha}), such that for all $x\in \mathcal{U}_{x^{*}}$ and $t\in[t_{0},\infty),$
\EQ \label{koonkoon1}(i):&&\hspace{.2cm}\alpha_{1}\left(d(x,x^{*})\right)\leq w(x,t)\leq\alpha_{2}\left(d(x,x^{*})\right),\nnum\\
 (ii):&&\hspace{.2cm}\mathcal{L}_{f(x,t)}w\leq -\alpha_{3}\left(d(x,x^{*})\right),\nnum\\ 
 (iii):&&\hspace{.2cm}||T_{x}w||\leq \alpha_{4}\left(d(x,x^{*})\right),\EN
where $d(\cdot,\cdot)$ is the Riemannian metric, $\mathcal{L}$ is the Lie derivative and $Tw:TM\rightarrow T\mathds{R}\simeq \mathds{R}\times \mathds{R}$ is the pushforward of $w$.  
\end{lemma}

Note that the lemma above holds for the time invariant gradient system $\dot{x}=-\sum^{n}_{i=1}\frac{a^{2}_{i}}{2}\nabla_{\frac{\partial}{\partial x_{j}}}J(x)\frac{\partial}{\partial x_{i}}$ since by Lemma \ref{kkir2} the gradient system is locally asymptotically  stable. In this case that the Lyapunov function $w$ is time invariant. By the results of Lemma \ref{kkir2} it has been shown that $J$ can be considered as a Lyapunov function. However, Lyapunov $w$ may not be necessary identical to $J$. One may show that items (i)-(iii) in Lemma \ref{lf1} locally hold for $J$ around $x^{*}$ when  Assumption \ref{as} is satisfied. Also note that for a compact manifold $M$, $||T_{x}f||$ is a bounded operator and the hypothesis of Lemma \ref{lf1} are satisfied for the extremum seeking algorithm (\ref{kk}) on compact manifolds. 

  The following theorem gives the stability of (\ref{per}), where the nominal system is  locally uniformly asymptotically stable.
  \begin{theorem}[\hspace{-.01cm}\cite{Taringoo100}]
  \label{t7}
  Let $x^{*}$ be an equilibrium of dynamical system $\dot{x}=f(x,t)$, which is locally uniformly asymptotically stable (see \cite{Kha}) on a neighborhood $\mathcal{N}_{x^{*}}\subset \mathcal{U}^{n}_{x^{*}}$ ($\mathcal{U}^{n}_{x^{*}}$ is a normal neighborhood around $x^{*}$). Assume the perturbed dynamical system (\ref{per}) is complete and the  Riemannian norm of the perturbation $h\in \mathfrak{X}(M\times \mathds{R})$ is bounded on $\mathcal{N}_{x^{*}}$, i.e. $||h(x,t)||_{g^{M}}\leq \delta,x\in \mathcal{N}_{x^{*}},t\in[t_{0},\infty)$. Then, for sufficiently small $\delta$, there exists a neighborhood $U_{x^{*}}$ and a function $\rho\in\mathcal{K}$, such that 
  \EQ \limsup_{t\rightarrow \infty} d(\Phi_{f+h}(t,t_{0},x_{0}),x^{*})\leq\rho(\delta),\hspace{.2cm}x_{0}\in U_{x^{*}}.\nnum\EN  
  \end{theorem}
  \begin{proof}
  A full version of the proof is given in \cite{Taringoo100}. However, for the completeness of the analysis we present a sketch of the proof in this paper. Following the results of Lemma \ref{lf1}, there exists $\mathcal{U}_{x^{*}}\subset\mathcal{N}_{x^{*}}$, such that (\ref{koonkoon1}) holds for a Lyapunov function $w$. By Lemma \ref{lf1}, there exists $\mathcal{U}_{x^{*}}$ and $\alpha_{3}\in\mathcal{K}$, such that
  \EQ &&\mathfrak{L}_{f+h}w=\mathfrak{L}_{f}w+\mathfrak{L}_{h}w\leq -\alpha_{3}(d(x,x^{*}))+\mathfrak{L}_{h}w,\hspace{.2cm}x\in \mathcal{U}_{x^{*}}.\nnum\EN 
	First we show that the neighborhood $\mathcal{U}_{x^{*}}$ can be shrunk, such that $\Phi_{f+h}(t,t_{0},x_{0})\in \mathcal{U}_{x^{*}},\hspace{.2cm}t\in[t_{0},\infty)$ provided $x_{0}\in \mathcal{U}_{x^{*}}$. By Lemma 6.12 in \cite{Lewis}, there exists a compact sublevel set $\mathcal{N}_{b,t_{0}}(x^{*})\subset\mathcal{U}_{x^{*}}$ where $x^{*}\in int(\mathcal{N}_{b,t_{0}}(x^{*}))$. 
Hence,  by the Shrinking Lemma \cite{Lee4} there exists a precompact neighborhood $\mathcal{W}_{x^{*}}$, such that, $\mathcal{W}_{x^{*}}\subset int(\mathcal{N}_{b,t_{0}}(x^{*}))\subset \mathcal{N}_{b,t_{0}}(x^{*})$, see \cite{Lee4}. Hence, $M-\mathcal{W}_{x^{*}}$ is a closed set and $ \mathcal{N}_{b,t_{0}}(x^{*})\bigcap (M-\mathcal{W}_{x^{*}})$ is a compact set (closed subsets of compact sets are compact). The continuity of $\alpha_{3}$ and $d(\cdot,x^{*})$ together with the compactness of $ \mathcal{N}_{b,t_{0}}(x^{*})\bigcap (M-\mathcal{W}_{x^{*}})$ imply the existence of the following parameter $\mathfrak{M}$,
  \EQ \mathfrak{M}\doteq \sup_{x\in \mathcal{N}_{b,t_{0}}(x^{*})\bigcap (M-\mathcal{W}_{x^{*}})}-\alpha_{3}(d(x,x^{*}))< 0.\nnum\EN
  Note that $\alpha_{3}\in \mathcal{K}$, $x\in \mathcal{N}_{b,t_{0}}(x^{*})\bigcap (M-\mathcal{W}_{x^{*}})$ and since $\mathcal{W}_{x^{*}}$ is a neighborhood of $x^{*}$ then $d(x,x^{*})>0,\hspace{.2cm}x\in \mathcal{N}_{b,t_{0}}(x^{*})\bigcap (M-\mathcal{W}_{x^{*}})$. Therefore, $\mathfrak{M}<0$. Also we have $\mathfrak{L}_{h}w=dw(h)\leq ||dw||\cdot||h||_{g^{M}}\leq \delta ||dw||$, where $||dw||$ is the induced norm of the linear operator $dw:TM\rightarrow \mathds{R}$. The smoothness of $w$ and compactness of $\mathcal{N}_{b,t_{0}}(x^{*})$ together imply $||dw||<\infty$. It is important to note that $||dw||$ is closely related to $||Tw||$ through the component of the Riemannian metric $g$. As is shown by Theorem \ref{lf1}, $||T_{x}w||\leq\alpha_{4}(d(x,x^{*}))$. Hence, the smoothness of $M$ and compactness of $\mathcal{N}_{b,t_{0}}(x^{*})$ imply that $||dw||<\infty$. Note that $||dw||$ is the norm of the linear operator $dw:T_{x}M\rightarrow \mathds{R}$. Hence, for sufficiently small $\delta$, we have $\mathfrak{L}_{f+h}w<0, x\in \mathcal{N}_{b,t_{0}}(x^{*})\bigcap(M-\mathcal{W}_{x^{*}})$. Therefore, the state trajectory $\Phi_{f+h}(t,t_{0},x_{0}),\hspace{.2cm}t\in[t_{0},\infty)$ stays in $\mathcal{U}_{x^{*}}$ for all $x_{0}\in int(\mathcal{N}_{b,t_{0}}(x^{*}))$.
  
   Without loss of generality,  we assume $\mathcal{U}_{x^{*}}=\exp_{x^{*}}B_{r_{1}}(0), r_{1}<i(x^{*})$. Note that for sufficiently small $r_{1}$, we have $\exp_{x^{*}}B_{r_{1}}(0)\subset \mathcal{N}_{b,t_{0}}(x^{*})$.  Then, by the results of Lemma \ref{lf1}, the variation of $w$ along $f+h$ is given by
  \EQ  \mathfrak{L}_{f+h}w&&=\mathfrak{L}_{f}w+\mathfrak{L}_{h}w\leq -\alpha_{3}(d(x,x^{*}))+\mathfrak{L}_{h}w\nnum\\&&=-\alpha_{3}(d(x,x^{*}))+dw(h(x,t))\nnum\\&&=-\alpha_{3}(d(x,x^{*}))+T_{x}w(h(x,t)) \nnum\\&&\leq -\alpha_{3}(d(x,x^{*}))+||T_{x}w||\cdot||h(x,t)||_{g^{M}}\nnum\\&&\leq -\alpha_{3}(d(x,x^{*}))+\delta \alpha_{4}(d(x,x^{*}))\nnum\\&&\leq -(1-\theta)\alpha_{3}(d(x,x^{*}))-\theta\alpha_{3}(d(x,x^{*}))\nnum\\&&+\delta \alpha_{4}(d(x,x^{*}))\leq -(1-\theta)\alpha_{3}(d(x,x^{*})),\hspace{.2cm}\nnum\\&&\mbox{if}\hspace{.2cm} \alpha^{-1}_{3}(\frac{\delta\alpha_{4}(r_{1})}{\theta})\leq d(x,x^{*})\leq r_{1},\nnum\EN
  for some $r_{1}<\iota(x)$, $0<\theta<1$. 
  
  Define $\eta\doteq \alpha_{2}(\alpha^{-1}_{3}(\frac{\delta\alpha_{4}(r_{1})}{\theta}))$, then $\{x\in M\,|\,d(x,x^{*})\leq \alpha^{-1}_{3}(\frac{\delta\alpha_{4}(r_{1})}{\theta})\}\subset\mathcal{N}_{t,\eta}=\{x\in \mathcal{U}_{x^{*}} |\,w(x,t)\leq \eta\}\subset \{x\in M|\,\alpha_{1}(d(x,x^{*}))\leq \eta\}$.
  Hence, solutions initialized in $\{x\in M\,|\,d(x,x^{*})\leq \alpha^{-1}_{3}(\frac{\delta\alpha_{4}(r_{1})}{\theta})\}$ remain in $\{x\in M|\,\alpha_{1}(d(x,x^{*}))\leq \eta\}$ since $\dot{w}<0$ for $x\in \mathcal{N}_{t,\eta}-\{x\in M\,|\,d(x,x^{*})\leq \alpha^{-1}_{3}(\frac{\delta\alpha_{4}(r_{1})}{\theta})\}$. This proves 
  \EQ \limsup_{t\rightarrow \infty} d(\Phi_{f+h}(t,t_{0},x_{0}),x^{*})&\leq& \alpha_{1}^{-1}\left(\alpha^{-1}_{3}\left(\frac{\delta\alpha_{4}(r_{1})}{\theta}\right)\right)\nnum\\&\doteq&\rho(\delta),\nnum\EN
  for any $x_{0}\in U_{x^{*}}\doteq \{x\in M\,|\,d(x,x^{*})< \alpha^{-1}_{3}(\frac{\delta\alpha_{4}(r_{1})}{\theta})\}\bigcap int(\mathcal{N}_{b,t_{0}}(\hat{x}))$. 
\end{proof}  
\bibliographystyle{plain}
\bibliography{HSCC1}

\begin{thebibliography}{10}

\bibitem{mar1}
R.~Abraham, J.~E. Marsden, and T.~S. Ratiu.
\newblock {\em Manifolds, Tensor Analysis, and Applications}.
\newblock Springer, 1988.

\bibitem{abs}
P.A. Absil, R.~Mahony, and R.~Sepulchre.
\newblock {\em Optimization Algorithms on Matrix Manifolds}.
\newblock Princeton University Press, 2007.

\bibitem{wit}
M.~Adachi.
\newblock {\em Embeddings and Immersions}.
\newblock American Mathematical Soc., 1993.

\bibitem{Agra}
A.~Agrachev and Y.~Sachkov.
\newblock {\em Control Theory from the Geometric Viewpoint}.
\newblock Springer, 2004.

\bibitem{Bak}
C.~G. Baker, P.-A. Absil, and K.~A. Gallivan.
\newblock An implicit trust-region method on {R}iemannian manifolds.
\newblock {\em IMA J. Numer. Anal.}, 28:665--689, 2008.

\bibitem{bloch}
A.~M. Bloch.
\newblock {\em Nonholonomic Mechanics and Control}.
\newblock Springer, 2000.

\bibitem{Lewis}
F.~Bullo and A.D. Lewis.
\newblock {\em Geometric Control of Mechanical Systems: Modeling, Analysis, and
  Design for Mechanical Control Systems}.
\newblock Springer, 2005.

\bibitem{hans}
H.~B. D\"{u}rr, M.~S. Stankovi\'{c}, C.~Ebenbauera, and K.~H. Johansson.
\newblock Lie bracket approximation of extremum seeking systems.
\newblock {\em Automatica}, (49):1538--1552, 2013.

\bibitem{hans2}
H.~B. D\"{u}rr, M.~S. Stankovi\'{c}, K.~H. Johansson, and C.~Ebenbauera.
\newblock Extremum seeking on submanifolds in {E}uclidean spaces.
\newblock {\em Automatica}, Provisionally accepted for Publication, 2014.

\bibitem{fre}
O.~Freifeld.
\newblock Statistics on manifolds with applications to modeling shape
  deformations.
\newblock {\em PhD thesis, The Division of Applied Mathematics at Brown
  University}, 2014.

\bibitem{fri}
P.~Frihauf, M.~Krsti\'{c}, and T.~Ba\v{s}ar.
\newblock Nash equilibrium seeking in non cooperative games.
\newblock {\em IEEE Trans. Automatic Control}, 57(5):1192--1207, 2012.

\bibitem{gabay}
D.~Gabay.
\newblock Minimizing a differentiable function over a differentiable manifold.
\newblock {\em Journal of Optimization Theory and Applications}, 37:177--219,
  1982.

\bibitem{guay}
M.~Guay, D.~Dochain, and M.~Perrier.
\newblock Adaptive extremum seeking control of nonlinear dynamic systems with
  parametric uncertainties.
\newblock {\em Automatica}, 39:1283--1293, 2003.

\bibitem{Guay2}
M.~Guay, D.~Dochain, and M.~Perrier.
\newblock Adaptive extremum seeking control of continuous stirred tank
  bioreactors with unknown growth kinetics.
\newblock {\em Automatica}, 40:881--888, 2004.

\bibitem{hel}
U.~Helmke, S.~Riardo, and J.~B. Yoshizawa.
\newblock Newton's algorithm in {E}uclidean jordan algbras, with applications
  to robotics.
\newblock {\em Communications in Information and Systems}, 3(2):283–297, 2002.

\bibitem{jost}
J.~Jost.
\newblock {\em Reimannian Geometry and Geometrical Analysis}.
\newblock Springer, 2004.

\bibitem{Kha}
H.~K. Khalil.
\newblock {\em Nonlinear Systems}.
\newblock Prentice Hall, 2002.

\bibitem{Khon}
S.~Z. Khong, D.~Ne\v{s}i\'{c}, C.~Manzie, and Y.~Tan.
\newblock Multidimensional global extremum seeking via the direct optimisation
  algorithm.
\newblock {\em Automatica}, 49(7):1970--1978, 2013.

\bibitem{Klin}
W.~P.~A. Klingenberg.
\newblock {\em Riemannian Geometry}.
\newblock de Gruyter Studies in Mathematics, 1995.

\bibitem{knap}
A.~W. Knapp.
\newblock {\em Lie groups Beyond an Introduction}.
\newblock Birkhauser, 1996.

\bibitem{krs1}
M.~Krsti\'{c} and H.~W. Wang.
\newblock Stability of extremum seeking feedback for general nonlinear dynamic
  systems.
\newblock {\em Automatica}, 36:595--601, 2000.

\bibitem{Lee3}
J.~M. Lee.
\newblock {\em Riemannian Manifolds, An Introduction to Curvature}.
\newblock Springer, 1997.

\bibitem{Lee4}
J.~M. Lee.
\newblock {\em Introduction to Topological Manifolds}.
\newblock Springer, 2000.

\bibitem{Lee2}
J.~M. Lee.
\newblock {\em Introduction to Smooth Manifolds}.
\newblock Springer, 2002.

\bibitem{Luenberger}
D.~G. Luenberger.
\newblock The gradient projection method along geodesics.
\newblock {\em Management Science}, 18(11):620--631, 1972.

\bibitem{man}
J.~H. Manton.
\newblock Optimisation algorithms exploting unitary constraints.
\newblock {\em IEEE Transactions on Signal Processing}, 50(3):635–650, 2002.

\bibitem{nes55}
D.~Ne\v{s}i\'{c}, A.~Mohammadi, and C.~Manzie.
\newblock A framework for extremum seeking control of systems with parameter
  uncertainties.
\newblock {\em IEEE Trans. Aut. Cont.}, 58(2):435--448, 2013.

\bibitem{nes6}
D.~Ne\v{s}i\'{c}, A.~Mohammadi, and C.~Manzie.
\newblock A systematic approach to extremum seeking based on parameter
  estimation.
\newblock In {\em IEEE Conf. Decision and Control}, pages 3902--3907, Atlanta,
  2010.

\bibitem{nes5}
D.~Ne\v{s}i\'{c}, Y.~Tan, W.~Moase, and C.~Manzie.
\newblock A unifying approach to extremum seeking: adaptive schemes based on
  estimation of derivatives.
\newblock In {\em IEEE Conf. Decision and Control}, pages 4625--4630, Atlanta,
  2010.

\bibitem{pen}
X.~Pennec.
\newblock Bi-invariant means on {L}ie groups with {C}artan-{S}chouten
  connections.
\newblock {\em Lecture Notes in Computer Science, Geometric Science of
  Information}, 8085:59--67, 2013.

\bibitem{Pet}
P.~Petersen.
\newblock {\em Riemannian Geometry}.
\newblock Springer, 1998.

\bibitem{ring}
W.~Ring and B.~Wirth.
\newblock Optimization methods on {R}iemannian manifolds and their application
  to shape space.
\newblock {\em SIAM journal on Control and Optimization}, 22(2):596--627, 2012.

\bibitem{rudin}
W.~Rudin.
\newblock {\em Real and Complex Analysis}.
\newblock New York: McGraw-Hill, 1974.

\bibitem{sand}
J.~A. Sanders and F.~Verhulst.
\newblock {\em Averaging Methods in Nonlinear Dynamical Systems}.
\newblock Springer, 1985.

\bibitem{shub}
B.~O. Shubert.
\newblock A sequential method seeking the global maximum of a function.
\newblock {\em SIAM Journal on Numerical Analysis}, 17(9):379–388, 1972.

\bibitem{smith}
S.~T. Smith.
\newblock Optimization techniques on {R}iemannian manifolds.
\newblock {\em Fields Institute Communications}, 3(3):113--135, 1994.

\bibitem{rom}
R.~G. Strongin and Y.~D. Sergeyev.
\newblock {\em Global Optimization with Non-Convex Constraints}.
\newblock Springer Science+Business Media Dordrecht, 2000.

\bibitem{tan}
Y.~Tan, D.~Ne\v{s}i\'{c}, and I.~M.~Y. Mareels.
\newblock On non-local stability properties of extremum seeking control.
\newblock {\em Automatica}, 42(6):889--903, 2006.

\bibitem{Taringoo100}
F.~Taringoo, P.~M. Dower, D.~Ne\v{s}i\'{c}, and Y.~Tan.
\newblock {\em A Local Characterization of Lyapunov Functions and Robust
  Stability of Perturbed Systems on Riemannian Manifolds,
  http://arxiv.org/abs/1311.0078}.
\newblock Submitted to \textit{Automatica}, 2013.

\bibitem{teel}
A.~R. Teel and D.~Popovi\'{c}.
\newblock Solving smooth and non-smooth multivariable extremum seeking problems
  by the methods of nonlinear programming.
\newblock In {\em American Control Conference}, pages 2394--2399, 2001.

\bibitem{udr}
C.~Udriste.
\newblock {\em Convex Functions and Optimization Methods on Riemannian
  Manifolds}.
\newblock Kluwer Academic Publishers, 1994.

\bibitem{Varad}
V.~Varadarajan.
\newblock {\em Lie groups, Lie algebras, and their representations}.
\newblock Springer, 1984.

\bibitem{yang}
Y.~Yang.
\newblock Globally convergent optimization algorithms on {R}iemannian
  manifolds: Uniform framework for unconstrained and constrained optimization.
\newblock {\em Journal of Optimization Theory and Applications},
  132(2):245--265, 2007.

\end{thebibliography}
\end{document}